\documentclass{article}
\usepackage{graphicx} 
\title{Mean-field limits for stochastic particle systems on dense graphs}
\date{\today}

\usepackage[T1]{fontenc}
\usepackage{amsthm}
\usepackage{dsfont}
\usepackage[utf8]{inputenc}
\usepackage{amsfonts}
\usepackage{amssymb}
\usepackage{amsmath}
\usepackage{mathrsfs}
\usepackage{stmaryrd}
\usepackage{bbm}
\usepackage{enumerate}
\usepackage{comment}
\usepackage{xcolor}
\usepackage[normalem]{ulem}
\usepackage{xfrac}

\usepackage[colorinlistoftodos]{todonotes}

\usepackage{hyperref}

\newcommand{\E}{\mathbb{E}}

\newcommand{\Lcal}{\mathcal{L}}
\newcommand{\N}{\mathbb{N}}
\newcommand{\dd}{\mathrm d}

\newtheorem{theorem}{Theorem}[section]
\newtheorem{lemma}[theorem]{Lemma}
\newtheorem{proposition}[theorem]{Proposition}
\newtheorem{corollary}[theorem]{Corollary}

\newtheorem{remark}[theorem]{Remark}
\newtheorem{example}{Example}
\newtheorem{assumption}{Assumption}
\begin{document}
\author{A. Koutsimpela\footnote{School of Applied Mathematical and Physical Sciences, National Technical University
of Athens, 15780 Athens, Greece}, E. Magnanini\footnote{WIAS, Anton-Wilhelm-Amo-Straße 39, 10717 Berlin, Germany.}}
\maketitle

\begin{abstract}
We study stochastic interacting particle systems whose interaction structure is described by dense weighted directed graphs converging to a graphon.
In the thermodynamic limit, we prove a law of large numbers for the empirical measure process and derive a deterministic nonlinear master equation describing the macroscopic evolution. 
The limiting equation retains the heterogeneous interaction structure of the microscopic system through the limiting graphon, allowing for spatially non-homogeneous behaviors such as localized or community-type interactions.

\end{abstract}

\noindent  \bigskip
\\
{\bf Keywords:} Interacting particle systems, monomer-exchange, thermodynamic limit, graphons, spatial interaction, empirical measure 
\\\\
{\bf AMS Subject Classification 2010:} 60K35;  35Q70; 82C22.


\section{Introduction}

Interacting particle systems provide a fundamental framework for the microscopic description of collective stochastic dynamics and their emergent macroscopic  behaviour. We consider stochastic particle systems, where only one particle jumps at a time (monomer-exchange) between sites according to local interaction rules while preserving the total mass, leading to a rich class of models including zero-range \cite{godreche2016coarsening, jatuviriyapornchai2016coarsening}, inclusion \cite{chleboun2023sizebiased}, coagulation processes \cite{AIM2, AIM1}, and more general misanthrope-type processes \cite{cocozza1985processus}. Such systems have been extensively studied in probability theory and statistical mechanics, both from the point of view of invariant measures (see e.g. \cite{chleboun2014condensation, armendariz2013zero}) and large-scale limits \cite{daipra2017, pakdaman2010fluid}.

In the classical mean-field setting, particles evolve on complete or homogeneous graphs, where each site interacts equally with all others. In the large-system limit, such stochastic particle systems are often described by nonlinear mean-field master equations governing the time evolution of the occupation distribution at a typical lattice site \cite{godreche2003dynamics,godreche2005dynamics, evans2014condensation}. A rigorous derivation of such mean-field equations for conservative misanthrope-type systems on complete graphs was obtained in \cite{grosskinsky2019derivation}, where convergence of empirical measures towards nonlinear exchange-driven growth equations was established under bilinear growth assumptions on the rates.
More recently, this framework was extended in \cite{Koutsimpela2025MeanField} to symmetric superlinear kernels. Related developments concerning tagged particles and size-biased empirical measures were obtained in \cite{SGAKtagged}, where the limiting dynamics was connected to the environment seen from a tagged particle. These works remain within the homogeneous mean-field setting, meaning that the interaction structure is independent of the spatial location of the particles. 

In the present work we replace the homogeneous mean-field interaction structure by an inhomogeneous one. Instead of assuming that every particle interacts with every other particle with the same intensity, we allow the interaction rates to depend on the spatial labels; this is achieved via a sequence of dense weighted directed graphs converging to a graphon (see, e.g.
\cite{Lovsz2012LargeNA} and the references therein).
This extension makes it possible to model heterogeneous interaction geometries, very often present in systems arising from applications. 
Related mean-field limits on dense graphon-based networks have recently been
studied in the context of non-conservative dynamics, including interacting
diffusions and McKean--Vlasov systems; see
\cite{Bet_Coppini_Nardi_2024,COPPINI2025104728,oliveira2018interacting}.



More specifically, in our model the microscopic dynamics is given by a general misanthrope-type process in which particles jump from a site $x$ to a site $y$ at rate
\[
q_L(x,y)c(\eta_x,\eta_y),
\]
where the weights $q_L(x,y)$ encode the geometry of the underlying interaction network (of size $L$) and converge, after rescaling, to a limiting graphon \(W\). Simultaneously, the interaction rates are allowed to depend, via the function $c(\cdot,\cdot)$, on the occupation numbers (resp. $\eta_x$ and $\eta_y)$ at the departure and arrival sites. This setting covers a broad family of conservative particle systems.

To each configuration we associate an empirical measure on the product space $[0,1]\times\mathbb N_0$, encoding the spatial label and the local occupation number. Under suitable assumptions on the function $c$ and on the convergence of the weighted graphs toward a graphon, we prove a law of large numbers for the empirical process in the thermodynamic limit. The limiting dynamics is deterministic and characterized as the unique weak solution of a nonlinear master equation whose coefficients depend explicitly on the graphon $W$.

The paper is organized as follows. In Sec.~\ref{sec:2} we introduce the model and we state the main result, Thm.~\ref{mainthm}, together with
other auxiliary statements.
Some examples of graphon geometries are also provided in Subsec.~\ref{subsec:examples}. Sec.~\ref{Sec3:proofs} is devoted to proofs. Specifically, Subsec~\ref{subsec:tightness} provides the tightness result, in Subsec.~\ref{subsec:id_limit} we identify the limiting equation and in Subsec.~\ref{subsec:uniqueness} we prove uniqueness. 


\paragraph{Technical contribution.}
Compared with the homogeneous mean-field setting, the graphon framework
introduces two additional analytical aspects. First, the limiting kernel
is only assumed to be measurable, so the nonlinear drift functional is not
continuous under weak convergence. We deal with this issue through a
continuous-kernel approximation in \hyperref[step4]{Step 4} in Subsec~\ref{subsec:id_limit}.

Second, uniqueness of the limiting equation requires a stability estimate for
the occupation-level marginals. Since the difference of two solutions is a
signed measure, rather than a function, standard $L^1$-based arguments do
not apply directly. We therefore represent the
occupation marginals with respect to a common dominating measure and derive
the stability estimate at the level of their densities (see Subsec.~\ref{subsec:uniqueness}). 

\section{Model, notation and main results} \label{sec:2}


\subsection{Model and notation}\label{Sec:settin}
For each $L\in\mathbb N$, let $G_L = (V_L,q_L)$
be a finite weighted directed graph with vertex set
$V_L := \{1,\dots,L\}$, 
and weights $q_L : V_L \times V_L \to [0,\infty)$ with $ q_L(x,x)=0.$ We assume that the interaction is in the dense mean-field regime, i.e. that the rescaled weights 
 \begin{equation}
     A^{(L)}_{xy}:=L q_L(x,y)
 \end{equation}
remain of order 1 as $L\to \infty,$ i.e. 
\begin{equation}\label{eq:C_W}
C_W:=\sup_L \sup\limits_{x,y}A^{(L)}_{xy}<\infty.
\end{equation}
\noindent
To each site $x\in V_L$ we associate the point 
\(
u_x := \frac{x}{L} \in [0,1]\ .
\)
We associate to the graph $G_L$ the step kernel 
$W_L : [0,1]^2 \to [0,\infty)$ defined by
\[
W_L(u_x,u_y) := L\,q_L(x,y)\] 
and 
\[ W_L(u,v):= W_L(u_x,u_y) \quad \text{if } (u,v) \in I_x \times I_y,
\]
where $(I_x)_{x=1}^L$ is the partition of $[0,1]$ given by
\[
I_x := \Big(\frac{x-1}{L}, \frac{x}{L}\Big].
\]

We work under the following assumption, which ensures that the graphon sequence admits a limit.
\medskip

\begin{assumption}[Graphon convergence]\label{assump:graphons1}
There exists a measurable and bounded function $W : [0,1]^2 \to [0,\infty)$ such that 
\begin{equation}\label{cutnorm}
\|W_L - W\|_{\square} \longrightarrow 0 \qquad \text{as } L\to\infty,
\end{equation}
where $\|\cdot\|_{\square}$ denotes the cut norm, 
\begin{equation}\label{eq:cut_norm}
\|W\|_{\square}
=
\sup_{\substack{\varphi,\psi:[0,1]\to[-1,1] \\ \text{measurable}}}
\left|
\int_0^1 \int_0^1 W(u,v)\,\varphi(u)\psi(v)\,du\,dv
\right|.
\end{equation}
\end{assumption}


\paragraph{Generator(s).} In the following, we consider stochastic particle systems $(\eta(t):t>0)$ on the graph $G_L$  and configurations are denoted by ${\eta} =(\eta_x :x\in V_L)$, where $\eta_x \in \N_0$ is the number of particles on vertex $x\in V_L.$ We consider systems with a fixed number of particles $N=\sum_{x\in V_L} \eta_x$ and we denote by $E_{L,N}
:=
\Bigl\{
\eta\in\mathbb N_0^{V_L}:
\sum_{x\in V_L}\eta_x=N
\Bigr\}$ the state space of all such configurations.\\
The dynamics of the process is defined by the infinitesimal generator 
\begin{equation}
	\label{eq:GenMis}
	(\mathcal{L}_Lg)({\eta}):=\sum_{x,y\in V_L}q_L(x,y)c(\eta_{x},\eta_{y})(g({\eta}^{x\rightarrow y})-g({\eta})) \ ,\quad g\in C_b (E_{L,N}) \ ,
\end{equation}
where $C_b(E_{L,N})$ denotes the space of bounded continuous functions on
$E_{L,N}$. Note that, since $E_{L,N}$ is finite, every real-valued function on
$E_{L,N}$ belongs to $C_b(E_{L,N})$.
Here, $\eta^{x\to y}$ denotes the configuration obtained by moving one particle from site $x$ to site $y$, namely
$
\eta_k^{x\to y}
:=
\eta_k-\delta_{k,x}+\delta_{k,y}$,
 $k\in V_L,
$
and the quantity
$
q_L(x,y)c(\eta_x,\eta_y)
$
is the corresponding jump rate.
To ensure that the process is non-degenerate, the jump rates satisfy
\begin{equation}\label{cassum}
	\left\{ \begin{array}{cl}
		c(0,l)=0\;&\mbox{for all }\ell\geq 0\\
		c(k,\ell)>0\;&\mbox{for all }k>0\;\mbox{and }\ell\geq 0.
	\end{array} \right.
\end{equation}
Our main further assumption on the dynamics is that the rates grow sublinearly, in the sense that they are bounded by a bilinear function:
\begin{assumption}[Bound on the rates]\label{assump2}
\begin{equation}
	\label{eq:lip}
	c(k,\ell)\leq C_c k (1+\ell) \quad\mbox{for constant }C_c >0\ .
\end{equation}
\end{assumption}
\noindent
The configuration $\eta$ can also be encoded in terms of the empirical measure 
\[
\pi^L(\eta)
:=
\frac{1}{L} \sum_{x=1}^L \delta_{(u_x,\eta_x)}
\in \mathcal P([0,1]\times \mathbb{N}_0),
\]
where 
$
\mathcal P([0,1]\times\mathbb N_0)$
denotes the space of probability measures on
$
[0,1]\times\mathbb N_0$.
In particular, for any Borel set $A\subset [0,1]$, and any $k\in N_0$, \[
\pi^L(\eta)\big(A\times\{k\}\big)
=
\frac{1}{L}\,
\#\Big\{ x\in V_L:\ u_x\in A,\ \eta_x= k \Big\},
\] 
that is, $\pi^L(\eta)(A\times\{k\})$ represents the fraction of vertices
whose rescaled labels belong to $A$ and whose occupation number equals $k$.
For the Markov process $(\eta^L(t))_{t\ge0}$ generated by $\mathcal L_L$,  we
write $\pi_t^L := \pi^L(\eta^L(t))$ for the corresponding empirical 
process, which takes values in
$\mathcal P\!\left([0,1]\times\mathbb N_0\right)$.
Now, for $\varphi \in C_b([0,1]\times \mathbb{N}_0)$, define
\begin{equation}\label{eq:pi_phi}
\langle \pi^L_t, \varphi \rangle
:=
\frac{1}{L}\sum_{x=1}^L \varphi(u_x,\eta_x(t)).
\end{equation}
With this notation, the generator \eqref{eq:GenMis} can be rewritten as
\begin{equation}
\label{eq:generator_empirical}
\mathcal{L}_L \langle \pi^L, \varphi \rangle
=
\iint_{([0,1]\times\mathbb N_0)^2} W_L(u,v)c(k,\ell)\,G_\varphi((u,k),(v,\ell))\,
\pi^L(du,dk)\,\pi^L(dv,d\ell),
\end{equation}
where
\[
G_\varphi((u,k),(v,\ell))
:=
\varphi(u,k-1)-\varphi(u,k)
+ \varphi(v,\ell+1)-\varphi(v,\ell).
\]
Notice that the first moment is conserved and is uniformly integrable:
\[
\sup_L \int_{[0,1]\times\mathbb N_0} k\,\pi^L_0(du,dk) = \frac{1}{L} \sum_{x=1}^L \eta_x(0) = \frac{1}{L} \sum_{x=1}^L \eta_x(t) =\frac{N}{L}< \rho.
\]
Now we introduce the following notation. 
\begin{itemize}
\item We denote by $\mathbb{P}^L$ the law of $\pi^L$ and by $\mathbb{E}^L$ the associated expectation. 
\item For $\varphi\in C_b([0,1]\times\mathbb N_0)$, we denote by
$
\mathbb P_{\varphi}^L$ the law of the process $
(\langle \pi_t^L,\varphi\rangle)_{t\ge0}$
on the path space \(D([0,\infty),\mathbb R)\).
Equivalently, 
$\mathbb P_\varphi^L$ is the pushforward of \(\mathbb P^L\) under the map $
\pi \mapsto \langle \pi_t,\varphi\rangle\bigr.
$.
\end{itemize}

We denote the second moment by 
\begin{equation}
m^L_2(t):=\E^L \left[ \frac{1}{L}\sum \limits_{x\in V_L} \eta^2_x(t) \right]
\end{equation}
and we further assume a uniform bound on the initial second moment.
\begin{assumption}[Bound on the second moment]\label{assump:3}
\begin{equation}\label{initialcon0}
m^L_2(0)\leq C_2, \; \text{ for all } L\geq 1,
\end{equation} 
\end{assumption}
For our main result we will consider the thermodynamic limit with density $\rho$, i.e.
\begin{equation}\label{thermo}
	L\to\infty ,\ N=N_L \to\infty \quad\mbox{such that}\quad N/L\to\rho\geq 0\ .
\end{equation} 
Under condition \eqref{thermo}, the sequence $N/L$ is uniformly bounded,
 hence we may assume that there exists $\rho>0$ such that
\begin{equation}\label{thermob}
	 N/L\leq \rho\quad\mbox{for all }L\geq 1\ .
\end{equation}
\paragraph{Initial condition.} We consider a sequence of random initial configurations
$
(\eta^L(0))_{L\ge1}$,
and assume that the associated empirical measures satisfy
\begin{equation}\label{init_conv}
\pi_0^L
\Rightarrow
\pi_0
\qquad
\text{in }
\mathcal P([0,1]\times\mathbb N_0) \qquad \text{in probability},
\end{equation}
 where $\pi_0$ is a (deterministic) element of $\mathcal{P}([0,1]\times\mathbb N_0)$ with 
\begin{equation}\label{initialcon0b}
  \int \limits_{[0,1]\times \N_0} k \pi_0(du,dk) =\rho \text{ and }
    \int \limits_{[0,1]\times \N_0} k^2 \pi_0(du,dk) <\infty  
\end{equation}


\paragraph{Topologies on the spaces.}
We equip $\mathcal{P}([0,1]\times\mathbb N_0)$ with a metric $\delta$ (such as the Prokhorov metric) that induces the weak topology on $\mathcal{P}([0,1]\times\mathbb N_0)$\footnote{We recall that the \emph{weak} topology on $\mathcal{P}([0,1]\times\mathbb N_0)$ is the smallest topologies that makes the maps $\mu \rightarrow \int_{[0,1]\times\mathbb N_0} f(x) \mu(\dd x)$ continuous for all $f \in C_{b}([0,1]\times\mathbb N_0)$.}. Instead, the trajectories of $\pi^L$, live in the space of right-continuous functions with left-limits (càdlàg). We denote such space by $D([0,T], \mathcal{P}([0,1]\times \mathbb{N}_0))$ and we equip it with the $J_1$ Skorokhod topology. 
\subsection{Main results} \label{subs:main_Res}

\begin{proposition}[Second moment bound] \label{sec_mom}
 Assume  Assumptions~\ref{assump:graphons1}-\ref{assump2}. Then, there exists a constant $C >0$ such that
\begin{equation}
    m_2^L (t)\leq (1+m^L_2(0) ) e^{Ct} \quad\mbox{for all }t\geq 0\mbox{ and }L\geq 1\ .
\end{equation}
If, in addition, Assumption~\ref{assump:3} holds, then  there exists a constant $C >0$ such that
\begin{equation}
    m_2^L (t)\leq C e^{Ct} \quad\mbox{for all }t\geq 0\mbox{ and }L\geq 1\ .
\end{equation}
 
\end{proposition}

\begin{proposition}[Tightness]\label{tightness}
 Assume  Assumptions~\ref{assump:graphons1}-\ref{assump2}-\ref{assump:3}. Then the sequence of measures $(\mathbb{P}_{\varphi}^L)_{L\geq 1}$ is tight.
\end{proposition}
\begin{theorem}[Law of large numbers]\label{mainthm}
Under Assumptions~\ref{assump:graphons1}-\ref{assump2}-\ref{assump:3}, for every $T>0$,
\begin{equation}\label{eq:main_Res}
\pi^L \;\Longrightarrow\; \pi
\quad \text{in } D([0,T], \mathcal{P}([0,1]\times \mathbb{N}_0)),
\end{equation}
where $\pi=(\pi_t)_{t\in[0,T]}$ is  the unique solution of the following equation 
$\varphi\in C_b([0,1]\times\mathbb N_0)$,
\begin{equation}\label{eq:weak_limit}
\langle \pi_t,\varphi\rangle
=
\langle \pi_0,\varphi\rangle
+
\int_0^t
\iint_{([0,1]\times\mathbb N_0)^2}
W(u,v)c(k,\ell)\,
G_\varphi((u,k),(v,\ell))\,
\pi_s(du,dk)\pi_s(dv,d\ell)\,ds,
\end{equation}
with initial condition $\pi_0$,  satisfying 
\(
\int_0^T \langle \pi_t,k^2\rangle\,dt<\infty
\) for every $T>0$.
Here,
\[
G_\varphi((u,k),(v,\ell))
:=
\varphi(u,k-1)-\varphi(u,k)
+\varphi(v,\ell+1)-\varphi(v,\ell),
\]
where by convention, \(\varphi(u,-1):=0\).
\end{theorem}

Note that the limiting equation is a nonlinear graphon-dependent master equation whose
coefficients encode the heterogeneous interaction structure of the underlying
dense graph sequence.

\begin{corollary}[Kernel representation of the limit]
Assume the hypotheses of Theorem \ref{mainthm}. Then, for every \(t\ge 0\), the first marginal of \(\pi_t\) is a Lebesgue measure on \([0,1]\). Hence there exists a measurable probability kernel
\[
u\mapsto p_t(u,\cdot)\in \mathcal P(\mathbb N_0)
\]
such that
\[
\pi_t(du,\{k\})=p_t(u,k)\,du,
\qquad k\in\mathbb N_0.
\]
In particular,
\[
p_t(u,k)\ge 0,
\qquad
\sum_{k\ge 0}p_t(u,k)=1
\quad\text{for a.e. }u\in[0,1].
\]
Moreover, for all  $u\in [0,1]$, and $ k\in \N_0$, in distributional form,  the family \(p_t\) satisfies the nonlinear master equation: 
\begin{align}
\label{eq:limit-p}
\partial_t p_t(u,k)
&=
p_t(u,k+1)\,\mu_t(u,k+1)
+\mathbf 1_{\{k\ge 1\}}\,p_t(u,k-1)\,\lambda_t(u,k-1)
\nonumber\\
&\quad
- p_t(u,k)\big(\mu_t(u,k)+\lambda_t(u,k)\big),
\end{align}
where the nonlinear loss and gain rates are given by
\begin{align}
\label{eq:mu-lambda}
\mu_t(u,k)
&:=
\int_0^1 \sum_{\ell\ge0}
W(u,v)c(k,\ell)p_t(v,\ell)\,dv,
\\
\lambda_t(u,k)
&:=
\int_0^1 \sum_{\ell\ge0}
W(v,u)c(\ell,k)p_t(v,\ell)\,dv.
\end{align}
\end{corollary}
The proof of the corollary above is omitted, as the identification of the marginal of $\pi_t$ as Lebesgue measure is directly provided by \hyperref[step4]{Step 4} in Subsec~\ref{subsec:id_limit}.

\begin{remark}[Marginal formulation]\label{rem:marginal-form}
Testing \eqref{eq:weak_limit} with
$\varphi(u,k)=\psi(u)\mathbf 1_{\{k=m\}}$,
where $\psi\in C_b([0,1])$, yields
\begin{align}
&\frac{d}{dt}\int_{[0,1]} \psi(u)\,\pi_t(du,\{m\})\\
&\hspace{1.5cm}=
\sum_{\ell\ge0}\int_{[0,1]^2}
W(u,v)c(m+1,\ell)\psi(u)\,
\pi_t(du,\{m+1\})\pi_t(dv,\{\ell\})
\nonumber\\
&\hspace{1.5cm}+
\mathbf 1_{\{m\ge1\}}
\sum_{\ell\ge0}\int_{[0,1]^2}
W(v,u)c(\ell,m-1)\psi(u)\,
\pi_t(du,\{m-1\})\pi_t(dv,\{\ell\})
\nonumber\\
&\hspace{1.5cm}-
\sum_{\ell\ge0}\int_{[0,1]^2}
W(u,v)c(m,\ell)\psi(u)\,
\pi_t(du,\{m\})\pi_t(dv,\{\ell\})
\nonumber\\
&\hspace{1.5cm}-
\sum_{\ell\ge0}\int_{[0,1]^2}
W(v,u)c(\ell,m)\psi(u)\,
\pi_t(du,\{m\})\pi_t(dv,\{\ell\}).
\label{eq:marginal_form}
\end{align}
\end{remark}



\begin{remark}[Homogeneous case]\label{rem:homogeneous}
If we work on the complete graph, i.e.
\[
q_L(x,y)=\frac{1}{L-1}\mathbf 1_{\{x\neq y\}},
\qquad x,y\in V_L,
\]
and $\pi_t(du,dk)=du\,p_k(t)$, then \eqref{eq:weak_limit}
reduces to the mean-field equation
\begin{align}\label{EDG}
\frac{dp_k(t)}{dt}
&=
\sum_{\ell\ge0}c(k+1,\ell)p_\ell(t)p_{k+1}(t)
+
\sum_{\ell\ge1}c(\ell,k-1)p_\ell(t)p_{k-1}(t)
\nonumber\\
&\quad-
\left(
\sum_{\ell\ge0}c(k,\ell)p_\ell(t)
+
\sum_{\ell\ge0}c(\ell,k)p_\ell(t)
\right)p_k(t),
\qquad k\ge0.
\end{align}
\end{remark}

\subsection{Examples}\label{subsec:examples}
In this section, we provide some examples of the limiting weighted graphon $W$. 
For each $L\in\mathbb N$, we define 
$$
q_L(x,y)
:=
\frac1L
W\!\left(\frac{x}{L},\frac{y}{L}\right)
\mathbf 1_{\{x\neq y\}},
\qquad x,y\in\{1,\dots,L\}.
$$
Consequently,  the associated step kernel is given by
$$
W_L(u,v)
=
Lq_L(x,y),
\qquad
(u,v)\in I_x\times I_y.
$$
In all examples below, by construction, $
\|W_L-W\|_{\square}\longrightarrow0
$.
\begin{example}[Homogeneous mean-field interaction]
The simplest and trivial case is the homogeneous one, where all spatial locations
interact with the same intensity. This is obtained by letting $
W(u,v):=1$, with
$u,v\in[0,1]$, and is treated in \cite{grosskinsky2019derivation}.
\end{example}

\begin{example}[Community formation.]
In this second example, the macroscopic interaction geometry encoded by $W$ has a community structure (see Figure~\ref{fig:two-community-graphon}).
Fix $0<a<b<1$. Let $\chi\in C([0,1];[0,1])$ satisfy
$$
\chi(u)=1 \quad \text{for } u\in[0,a],
\qquad
\chi(u)=0 \quad \text{for } u\in[b,1].
$$
Define the continuous kernel
$$
W(u,v)
:=
\gamma
+
(\alpha-\gamma)\chi(u)\chi(v)
+
(\beta-\gamma)(1-\chi(u))(1-\chi(v))
$$
with parameters
$$
\alpha,\beta,\gamma>0,
\qquad
\gamma<\min\{\alpha,\beta\}.
$$
\end{example}

\begin{figure}[ht]
\centering
\includegraphics[width=\textwidth]{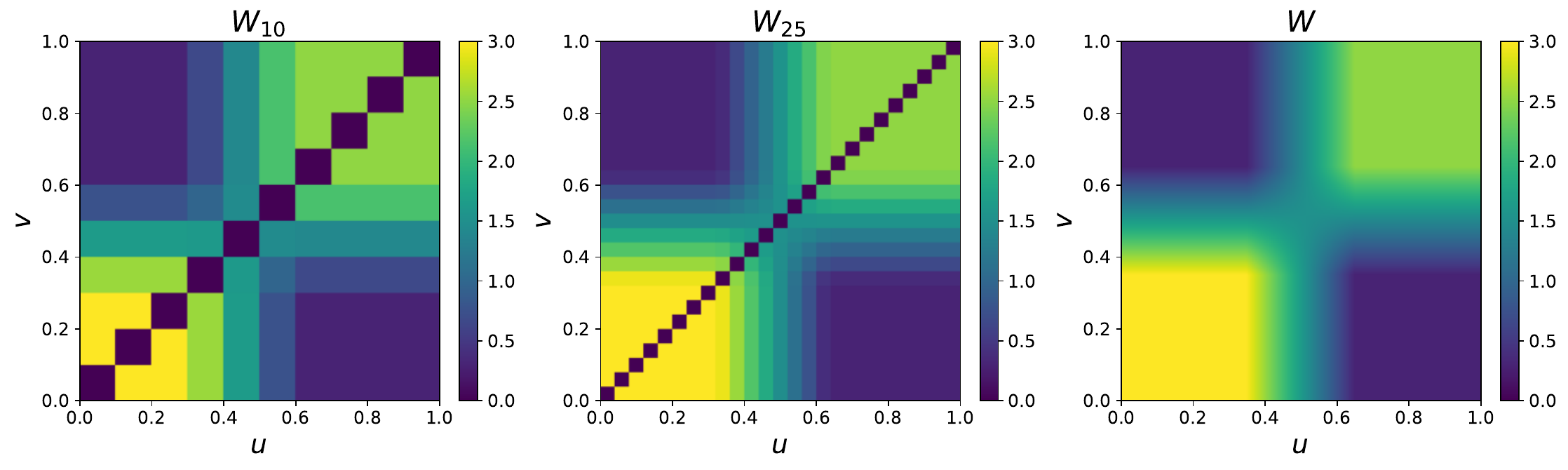}
\caption{
From left to right: the step graphons $W_{10}$ and $W_{25}$,
followed by the limiting continuous graphon $W$. The parameters used for the cut-off $\chi$ are
$a=0.35$ and
$b=0.65$. Additionally we used
$\alpha=3$,
$\beta=2.5$, and
$\gamma=0.3$.
As $L\to\infty$, the step kernels converge to the limiting graphon,
preserving the underlying community structure.
}
\label{fig:two-community-graphon}
\end{figure}

\begin{example}[Core-periphery]
In this third example the spatial variable creates regions with different macroscopic
connectivity levels: a highly interacting core and a weakly interacting
periphery.
Let $r\in C([0,1];(0,\infty))$ be a decreasing function. For instance, one may take
$$
r(u):=1+\theta(1-u),
\qquad \theta>0.
$$
Fix $a,b>0$, and define
$$
W(u,v):=a+b\,r(u)r(v),
\qquad u,v\in[0,1].
$$
Since $r$ is decreasing, the interaction kernel $W(u,v)$ is larger when
$u$ is close to $0$. Consequently, the left part of the interval
acts as a highly connected core, while the right part behaves as a weakly
interacting periphery (see Figure~\ref{fig:core-periphery-graphon}).
\end{example}
\begin{example}
(Directed interaction). In this example the interaction structure is asymmetric and induces a preferred direction of particle transfer. Fix parameters \(a,b>0\), and define
\[
W(u,v):=a+b\,\mathbf{1}_{\{u<v\}}, \qquad u,v\in[0,1].
\]
The resulting graphon is non-symmetric, since
\[
W(u,v)\neq W(v,u)
\]
whenever \(u\neq v\). Interactions from sites with smaller spatial labels toward sites with larger labels occur with higher intensity than in the opposite direction. Consequently, the limiting dynamics exhibits a macroscopic directional bias from left to right (see Fig.~\ref{fig:directed-graphon}).
\end{example}
\begin{figure}[ht]
\centering
\includegraphics[width=\textwidth]{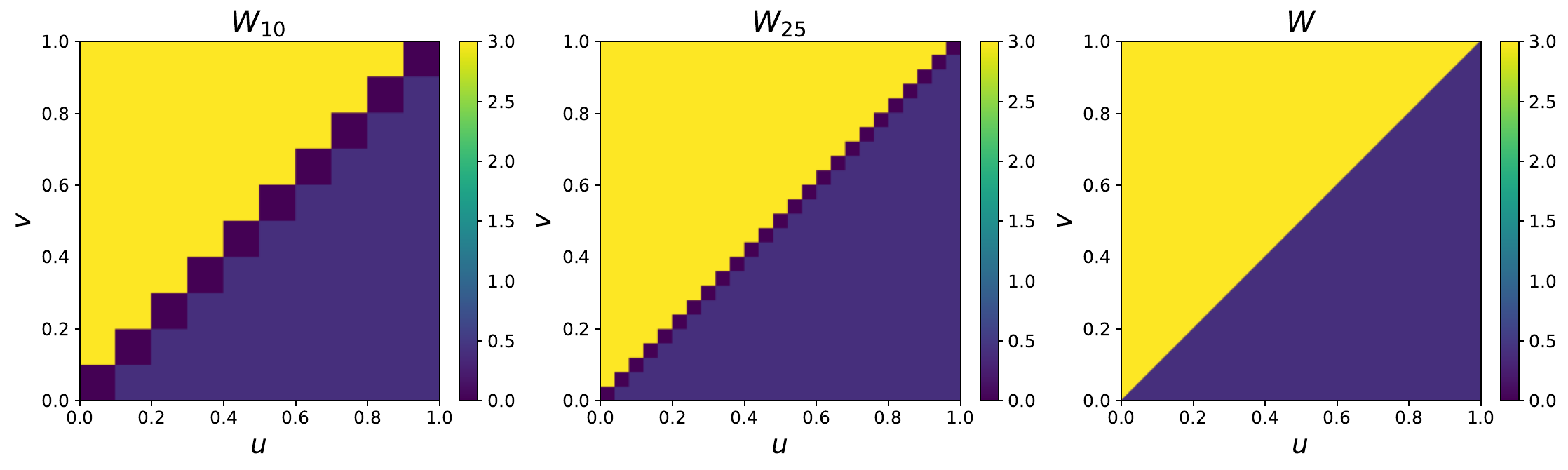}
\caption{
Directed interaction kernels with parameters $
a=0.3$,
$b=0.7$.
From left to right: the step graphons $W_{10}$ and $W_{25}$,
followed by the limiting continuous graphon $W$. The graphon assigns a larger interaction intensity to ordered pairs
\((u,v)\) with \(u<v\) than to pairs with \(u>v\).
}
\label{fig:directed-graphon}
\end{figure}

\begin{figure}[ht]
\centering
\includegraphics[width=\textwidth]{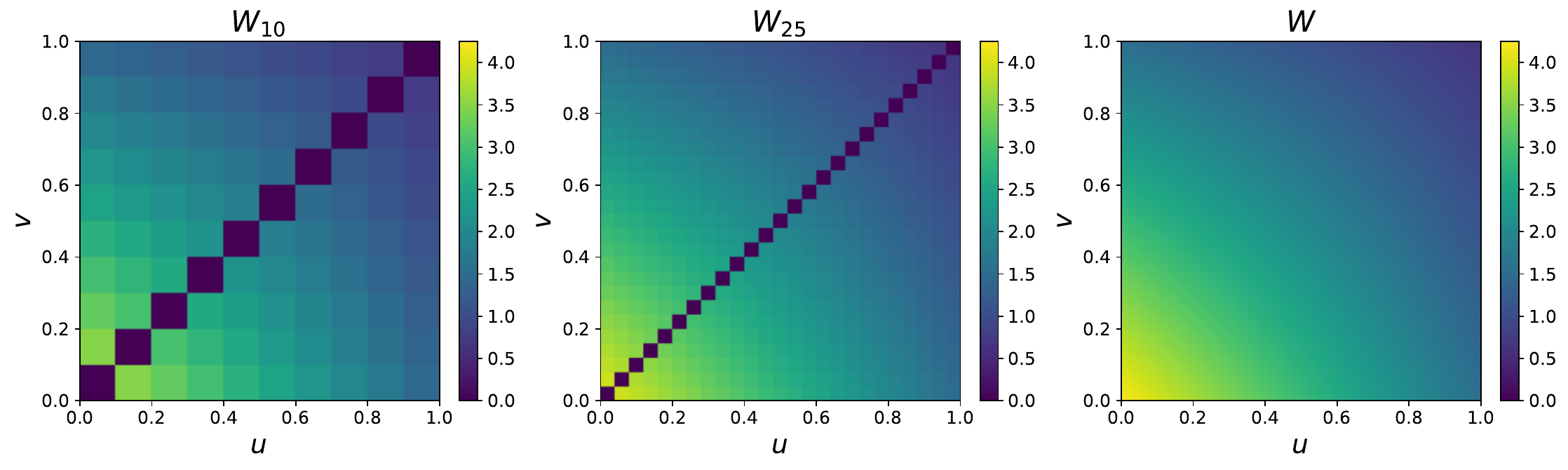}
\caption{
Core--periphery interaction kernels with parameters $
a=0.2$,
$b=0.45$, and
$\theta=2$.
From left to right: the step graphons $W_{10}$ and $W_{25}$,
followed by the limiting continuous graphon $W$.
The region near $u=0$ represents the core and has stronger interactions,
while the region near $u=1$ represents the periphery.
}
\label{fig:core-periphery-graphon}
\end{figure}

\section{Proofs}\label{Sec3:proofs}

\subsection{Second moment bound}
\begin{proof}[Proof of Prop.~\ref{sec_mom}]
For \(x\in V_L\), choosing \(g(\eta)=\eta_x^2\) in the generator
\eqref{eq:GenMis}, we obtain
\[
\mathcal L_L \eta_x^2
=
\sum_{y\in V_L}
\Big(
q_L(y,x)c(\eta_y,\eta_x)(1+2\eta_x)
+
q_L(x,y)c(\eta_x,\eta_y)(1-2\eta_x)
\Big).
\]
Using the bounds \(q_L(x,y)\le C_W/L\) (recall \eqref{eq:C_W}),
\(c(k,\ell)\le C_c\,k(1+\ell)\), and \(1-2\eta_x\le 1\), we get
\[
\begin{aligned}
\frac{d}{dt}m_2^L(t)
&=
\frac1L\sum_{x\in V_L}
\E^L[\mathcal L_L\eta_x^2(t)]
\\
&\le
\frac{C_WC_c}{L^2}
\sum_{x,y\in V_L}
\E^L\!\Big[
\eta_y(t)\bigl(1+\eta_x(t)\bigr)
\bigl(1+2\eta_x(t)\bigr)
+
\eta_x(t)\bigl(1+\eta_y(t)\bigr)
\Big]\\
&=2C_WC_c
\frac1{L^2}
\sum_{x,y\in V_L}
\E^L\!\Big[
\eta_x(t)\bigl(1+2\eta_y(t)+\eta_y^2(t)\bigr)
\Big]
\end{aligned}
\]
Using conservation of the total mass,
$
\frac1L\sum_{x\in V_L}\eta_x(t)
=
\frac1L\sum_{x\in V_L}\eta_x(0)
\le \rho ,
$
we deduce
\[
\begin{aligned}
\frac{d}{dt}m_2^L(t)
&\le
2C_WC_c
\E^L\!\left[
\left(\frac1L\sum_{x\in V_L}\eta_x(t)\right)
\left(
1+
2\frac1L\sum_{y\in V_L}\eta_y(t)
+
\frac1L\sum_{y\in V_L}\eta_y^2(t)
\right)
\right]
\\
&\le
2C_WC_c\,\rho
\left(
1+2\rho+m_2^L(t)
\right)
\\
&\le
\bar C\bigl(1+m_2^L(t)\bigr),
\end{aligned}
\]
for a suitable constant \(\bar C>0\). Hence
$
\frac{d}{dt}\bigl(1+m_2^L(t)\bigr)
\le
\bar C\,\bigl(1+m_2^L(t)\bigr),
$
and Gronwall's lemma yields
\[
m_2^L(t)
\le
\bigl(1+m_2^L(0)\bigr)e^{\bar C t}
\qquad\quad\mbox{for all }t\geq 0\mbox{ and }L\geq 1.
\]
\end{proof}

\subsection{Tightness}\label{subsec:tightness}

\begin{proof}[Proof of Prop.~\ref{tightness}]
Fix such a test function \(\varphi\in C_b([0,1]\times\mathbb N_0)\), and define
\begin{equation}\label{eq:Def:Phi}
\Phi_t^L
:=
\langle \pi_t^L,\varphi\rangle
=
\int_{[0,1]\times\mathbb N_0}
\varphi(u,k)\,\pi_t^L(du,dk).
\end{equation}
To establish tightness of $
(\mathbb P_\varphi^L)_{L\geq 1},
$ we use a version of Aldous's criterion (see Theorem 16.10 in \cite{billingsley2013convergence}).
We show that, for every \(t\ge0\),
\begin{equation}
\label{eq:tight1}
\lim_{a\to\infty}
\limsup_{L\to\infty}
\mathbb P^L\bigl(|\Phi_t^L|\ge a\bigr)
=0,
\end{equation}
and that for every \(\varepsilon>0\) and \(t>0\),
\begin{equation}
\label{eq:tightness}
\lim_{\delta_0\downarrow0}
\limsup_{L\to\infty}
\sup_{0<\delta\le\delta_0}
\sup_{\tau\in\mathfrak T_t}
\mathbb P^L
\bigl(
|\Phi_{\tau+\delta}^L-\Phi_\tau^L|>\varepsilon
\bigr)
=0,
\end{equation}
where \(\mathfrak T_t\) denotes the set of stopping times bounded by \(t\).

Since $\pi_t^L$ is a probability measure and
$\varphi$ is bounded, we immediately get $|\Phi_t^L|\leq \|\varphi\|_\infty$
for all \(t\ge0\) and \(L\ge1\). Hence \eqref{eq:tight1} follows
 by Markov's inequality,
\[
\mathbb P^L\bigl(|\Phi_t^L|\ge a\bigr)
\le
\frac{\mathbb E^L\bigl[|\Phi_t^L|\bigr]}{a}
\le
\frac{\|\varphi\|_\infty}{a},
\]
where $\mathbb{E}_{\varphi}^L$ denotes the expectation w.r.t. $\mathbb{P}_\varphi^L$.
Now, to prove \eqref{eq:tightness}, fix \(\delta_0>0\), \(\tau\in\mathfrak T_t\), and let
\(0<\delta\le \delta_0\).
Since \((\pi_s^L)_{s\ge0}\) is a Markov process and
\(\langle \pi,\varphi\rangle\) is bounded and measurable,
Dynkin's formula yields
\begin{equation}
\label{eq:itoetaMIM}
\Phi_{\tau+\delta}^L-\Phi_\tau^L
=
\int_\tau^{\tau+\delta}
(\mathcal L_L\Phi^L)(\pi_s^L)\,ds
+
M_\Phi^L(\tau+\delta)-M_\Phi^L(\tau),
\end{equation}
where \((M_\Phi^L(s))_{s\ge0}\) is a martingale whose quadratic variation is obtained by integrating the carré du champ operator:
\begin{equation}
\label{eq:mart}
[M_\Phi^L](t)
=
\int_0^t
\Big(
\mathcal L_L (\Phi^L)^2
-
2\Phi^L\,\mathcal L_L\Phi^L
\Big)(\pi_s^L)\,ds .
\end{equation}
To compute the drift, it is convenient to introduce the notation 
\[
F_\varphi^L(\eta)
:=
\langle \pi^L(\eta),\varphi\rangle
=
\frac1L\sum_{x\in V_L}\varphi(u_x^L,\eta_x).
\]
Then,
\[
\begin{aligned}
(\mathcal L_L F_\varphi^L)(\eta)
&=
\frac1L
\sum_{x,y\in V_L}
q_L(x,y)c(\eta_x,\eta_y)
\\
&\qquad\qquad\times
\Big[
\varphi(u_x^L,\eta_x-1)
+\varphi(u_y^L,\eta_y+1)
-\varphi(u_x^L,\eta_x)
-\varphi(u_y^L,\eta_y)
\Big],
\end{aligned}
\]
where again we used the convention that $\varphi(u,-1)=0$. 
Thus, for all \(\eta\in E_{L,N}\),
\begin{equation}\label{eq:absv_gen}
\begin{aligned}
\big|(\mathcal L_L F_\varphi^L)(\eta)\big|
&\overset{\eqref{eq:C_W},\eqref{eq:lip}}{\le}
4C_cC_W\|\varphi\|_\infty
\frac1{L^2}
\sum_{x,y\in V_L}\eta_x(1+\eta_y)
\\
&=
4C_cC_W\|\varphi\|_\infty
\left(\frac1L\sum_{x\in V_L}\eta_x\right)
\left(1+\frac1L\sum_{y\in V_L}\eta_y\right)
\\
&\le
4C_cC_W\|\varphi\|_\infty\,\rho(1+\rho),
\end{aligned}
\end{equation}
where \(N=N_L=\sum_{x\in V_L}\eta_x\) is the (constant) number of particles and \(N/L\le\rho\).
In order to prove \eqref{eq:tightness}, we again aim to use Markov's inequality. It is therefore enough to produce a bound on
$
\E^L\Big[
\big|\Phi_{\tau+\delta}^L-\Phi_\tau^L\big|
\Big]$
which is uniform in $L$, $\tau\in\mathfrak T_t$, and $0<\delta\le\delta_0$.
Taking the absolute value in  \eqref{eq:itoetaMIM}, for \(0<\delta\le \delta_0\), we obtain
\begin{align}
\label{toboundMim}
\E^L\Big[
\big|\Phi_{\tau+\delta}^L-\Phi_\tau^L\big|
\Big]
&\le
\E^L\left[
\int_\tau^{\tau+\delta}
\big|(\mathcal L_LF_\varphi^L)(\eta(s))\big|\,ds
\right]
\nonumber\\
&\quad+
\E^L\left[
\big|M_\Phi^L(\tau+\delta)-M_\Phi^L(\tau)\big|
\right]
\nonumber\\
&\overset{\eqref{eq:absv_gen}}{\le}
4C_cC_W\|\varphi\|_\infty \rho(1+\rho)\,\delta_0
\nonumber\\
&\quad+
\left(
\E^L\left[
[M_\Phi^L](\tau+\delta)-[M_\Phi^L](\tau)
\right]
\right)^{1/2}.
\end{align}
For the second term, we used Cauchy--Schwarz inequality together with the identity
$
\E^L\!\left[
\bigl(M_\Phi^L(\tau+\delta)-M_\Phi^L(\tau)\bigr)^2
\right]
=
\E^L\!\left[
[M_\Phi^L](\tau+\delta)-[M_\Phi^L](\tau)
\right]$,
which follows from the optional stopping theorem applied to the martingale
$
(M_\Phi^L(t))^2-[M_\Phi^L](t).
$
Then, to control the last term in \eqref{toboundMim}, it remains to bound
the carré du champ operator. For all \(\eta\in E_{L,N}\), we have
\begin{align}
&\Big(
\mathcal L_L(F_\varphi^L)^2
-
2F_\varphi^L\,\mathcal L_LF_\varphi^L
\Big)(\eta)
\nonumber\\
&=
\frac1{L^2}
\sum_{x,y\in V_L}
q_L(x,y)c(\eta_x,\eta_y)
\Big[
\varphi(u_x^L,\eta_x-1)
+\varphi(u_y^L,\eta_y+1)
-\varphi(u_x^L,\eta_x)
-\varphi(u_y^L,\eta_y)
\Big]^2
\nonumber\\
&\le
16\|\varphi\|_\infty^2
\frac{C_cC_W}{L^3}
\sum_{x,y\in V_L}
\eta_x(1+\eta_y)
=
\frac{16\|\varphi\|_\infty^2C_cC_W}{L}
\left(\frac1L\sum_{x\in V_L}\eta_x\right)
\left(1+\frac1L\sum_{y\in V_L}\eta_y\right)
\nonumber\\
&\le
\frac{16\|\varphi\|_\infty^2C_cC_W}{L}
\rho(1+\rho).
\label{vanish}
\end{align}
Therefore,
\begin{align}\label{eq:stima_qv}
\E^L \Big[
[M_\Phi^L](\tau+\delta)-[M_\Phi^L](\tau)
\Big]
&\leq
\delta_0\,
\frac{16\|\varphi\|_\infty^2 C_cC_W}{L}
\,\rho(1+\rho).
\end{align}
Combining \eqref{toboundMim} and \eqref{eq:stima_qv}, we obtain
\[
\E^L\Big[
\big|\Phi_{\tau+\delta}^L-\Phi_\tau^L\big|
\Big]
\le
4C_cC_W\|\varphi\|_\infty \rho(1+\rho)\,\delta_0
+
4\|\varphi\|_\infty
\sqrt{
\frac{C_cC_W\rho(1+\rho)}{L}\,\delta_0
}.
\]
The right-hand side is uniform in
\(\tau\in\mathfrak T_t\) and \(0<\delta\le\delta_0\).
Hence, by Markov's inequality,
\[
\sup_{0<\delta\le\delta_0}
\sup_{\tau\in\mathfrak T_t}
\mathbb P^L
\Big(
|\Phi_{\tau+\delta}^L-\Phi_\tau^L|>\varepsilon
\Big)
\le
\frac{\widehat{C}}{\varepsilon}
\left(
\delta_0
+
\sqrt{\frac{\delta_0}{L}}
\right),
\]
for a constant \(\widehat{C}\equiv\widehat{C}(\varphi, C_c, C_W, \rho)<\infty\) depending only on
\(\varphi\), \(C_c\), \(C_W\), and \(\rho\).
Taking first \(\limsup_{L\to\infty}\) and then letting
\(\delta_0\downarrow0\), we obtain \eqref{eq:tightness}. This completes the proof.
\end{proof}

\subsection{Identification of the limit}\label{subsec:id_limit}
By Proposition \ref{tightness}, for every continuous bounded  test function
\(\phi : [0,1]\times\mathbb N_0 \to \mathbb R\), the sequence
\(
\bigl(\langle \pi_t^L,\phi\rangle\bigr)_{t\in[0,T]}
\)
is tight in \(D([0,T],\mathbb R)\). Since \([0,1]\times\mathbb N_0\) is Polish, it follows
that
\((\pi^L)_{L\geq 1}\) is tight in \(
D([0,T],\mathcal P([0,1]\times\mathbb N_0))
\) (see Ch.3 \cite{ethier2009markov}).  Therefore, by Prohorov's theorem, every subsequence admits a further
subsequence, still denoted by $(\pi^L)_{L\geq 1}$ converging in distribution to a limit process
\(
(\pi_t)_{t\in[0,T]},
\)
which may still be random at this stage, with law $\mathbb Q$.
By the Skorokhod's representation theorem~\cite[Theorem~IV.13, page 71]{pollard}, there exists a common probability
space $(\hat{\Omega},\hat{\mathcal{F}},\hat{\mathbb{P}})$ and stochastic processes $\hat{\pi}^{L},
\hat{\pi}
:
\hat{\Omega}
\to
D\bigl([0,T],\mathcal P([0,1]\times\mathbb N_0)\bigr)$
with laws $\mathbb{P}^{L}$ and $\mathbb{Q}$ such that 
\[
\hat{\pi}^L \longrightarrow \hat{\pi}
\qquad \hat{\mathbb{P}}-\text{almost surely in }
D([0,T],\mathcal P([0,1]\times\mathbb N_0)),
\]
with respect to the $J_1$ topology.
In the remainder of the proof we work with this realization and we omit the hats from the notation.
In particular, for almost every $\omega$ and Lebesgue-a.e.
$s\in[0,T]$,
\begin{equation}\label{eq:weak_conv_pi}
\pi_s^L(\omega)\Rightarrow \pi_s(\omega) \qquad \text{in } \mathcal P([0,1]\times\mathbb N_0).
\end{equation}
Now, the goal is to study the limit $L\to\infty$ of Eq.~\eqref{eq:itoetaMIM} (combined with the definitions in \eqref{eq:Def:Phi} and \eqref{eq:generator_empirical}):
\begin{equation}
\label{eq:weakL}
\begin{aligned} 
\langle \pi_t^L,\varphi\rangle
&=
\langle \pi_0^L,\varphi\rangle
\\
&\quad+
\int_0^t
\iint_{E^2}
W_L(u,v)c(k,\ell)
G_{\varphi}((u,k),(v,\ell))
\,\pi_s^L(du,dk)\,\pi_s^L(dv,d\ell)
\,ds
\\
&\quad+
M_\Phi^L(t).
\end{aligned}
\end{equation}

First we address the middle term in \eqref{eq:weakL} and we show that

\begin{equation}\label{eq:L1_goal}
\begin{aligned}
    \lim \limits_{L\to \infty} \E \bigg[ \int_0^t & \bigg|  \iint_{([0,1]\times\mathbb N_0)^2} W_L(u,v)c(k,\ell)G_{\varphi}((u,k),(v,\ell)) \pi^L_s(du,dk)  \pi^L_s(dv,d\ell) \nonumber \\ & -\iint W(u,v)c(k,\ell)G_{\varphi}((u,k),(v,\ell)) \pi_s(du,dk)  \pi_s(dv,d\ell) \bigg| ds\bigg] =0.
\end{aligned}
\end{equation}
We will then collect the contribution of all terms in \hyperref[step5]{Step 5}.
\paragraph{Step 1. Truncation in the occupation variables.}
Fix $K\in\mathbb N$ and define
$$
G_\varphi^K((u,k),(v,\ell))
:=
G_\varphi((u,k),(v,\ell))\,\mathbf 1_{\{k\le K,\ \ell\le K\}}.
$$
Set \(E:=[0,1]\times\mathbb N_0\). Then by triangle inequality we deduce:
\begin{equation}\label{triang:ineq}
\begin{aligned}
&\int_0^t
\left|
\iint_{E^2}
W_L(u,v)c(k,\ell)G_\varphi((u,k),(v,\ell))
\,\pi_s^L(du,dk)\pi_s^L(dv,d\ell)
\right.
\\
&\hspace{2.7cm}\left.
-
\iint_{E^2}
W(u,v)c(k,\ell)G_\varphi((u,k),(v,\ell))
\,\pi_s(du,dk)\pi_s(dv,d\ell)
\right|ds
\\
&\le
\int_0^t
\left|
\iint_{E^2}
W_L(u,v)c(k,\ell)G_\varphi^K((u,k),(v,\ell))
\,\pi_s^L(du,dk)\pi_s^L(dv,d\ell)
\right.
\\
&\hspace{2.7cm}\left.
-
\iint_{E^2}
W(u,v)c(k,\ell)G_\varphi^K((u,k),(v,\ell))
\,\pi_s(du,dk)\pi_s(dv,d\ell)
\right|ds
\\
&\quad+
\int_0^t
\left|
\iint_{E^2}
W_L(u,v)c(k,\ell)
\bigl(G_\varphi-G_\varphi^K\bigr)((u,k),(v,\ell))
\,\pi_s^L(du,dk)\pi_s^L(dv,d\ell)
\right|ds
\\
&\quad+
\int_0^t
\left|
\iint_{E^2}
W(u,v)c(k,\ell)
\bigl(G_\varphi-G_\varphi^K\bigr)((u,k),(v,\ell))
\,\pi_s(du,dk)\pi_s(dv,d\ell)
\right|ds .
\end{aligned}
\end{equation}
In the equation display above there are three terms. We discuss them separately.
\begin{itemize}
\item Passing to the expectation, for the second term in \eqref{triang:ineq} we claim that we  have:
\begin{equation}
\begin{aligned}
\label{eq:tail-prelimit}
&\lim_{K\to\infty}\sup_L\,
\mathbb E\Bigg[
\int_0^t
\Bigg|
\iint W_L(u,v)\,c(k,\ell)\,\big(G_\varphi-G_\varphi^K\big)((u,k),(v,\ell))\\
&\hspace{7cm}\times d(\pi_s^L\otimes\pi_s^L)
\Bigg|ds
\Bigg]=0.
\end{aligned}
\end{equation}
Indeed, 
$
G_\varphi-G_\varphi^K
=
G_\varphi\,\mathbf 1_{\{k>K\}\cup\{\ell>K\}},
$
and, since $\varphi$ is bounded, it holds
$
|G_\varphi-G_\varphi^K|
\le
4\|\varphi\|_\infty
\mathbf 1_{\{k>K\}\cup\{\ell>K\}}.
$
Using \eqref{eq:lip}, the boundedness of $W_L$, and
$
\mathbf 1_{\{k>K\}}
\leq \frac{k}{K}$,
$\mathbf 1_{\{\ell>K\}}
\le \frac{\ell}{K}$,
we obtain
\begin{equation}\label{eq:bound:absv}
\mathbb E\left|
\iint_{E^2} W_L\,c\,(G_\varphi-G_\varphi^K)
\,d(\pi_s^L\otimes\pi_s^L)
\right|
\le
\frac{C}{K}\bigl(1+m_2^L(s)\bigr),
\end{equation}
for a constant \(C\) independent of \(L\), \(K\), and \(s\). Therefore, by Proposition~\ref{sec_mom},
\[
\sup_L
\mathbb E\!\left[
\int_0^t
\left|
\iint_{E^2} W_L\,c\,(G_\varphi-G_\varphi^K)
\,d(\pi_s^L\otimes\pi_s^L)
\right|
ds
\right]
\le
\frac{C_t}{K},
\]
which yields \eqref{eq:tail-prelimit} by letting \(K\to\infty\).

\item 
Similarly, by Proposition~\ref{sec_mom} and Fatou's lemma, 
for the third term in \eqref{triang:ineq}
we have
\begin{equation}
\label{eq:tail-limit}
\lim_{K\to\infty}
\mathbb E\left[
\int_0^t
\left|
\iint_{E^2} W(u,v)\,c(k,\ell)\,\big(G_\varphi-G_\varphi^K\big)((u,k),(v,\ell))
\,d(\pi_s\otimes\pi_s)
\right|ds
\right]
=0.
\end{equation}
\end{itemize}
We finally discuss the limit of the expectation of the first term in \eqref{triang:ineq}.
First, we prove that for every fixed $K\in\mathbb N$,
\begin{equation}
\begin{aligned}
\label{eq:truncated-drift-goal}
&\lim_{L\to\infty}
\mathbb E\Bigg[
\int_0^t
\Bigg|
\iint_{E^2} W_L(u,v)\,c(k,\ell)\,G_\varphi^K((u,k),(v,\ell))\,
\pi_s^L(du,dk)\,\pi_s^L(dv,d\ell)\\
&-
\iint_{E^2} W(u,v)\,c(k,\ell)\,G_\varphi^K((u,k),(v,\ell))\,
\pi_s(du,dk)\,\pi_s(dv,d\ell)
\Bigg|ds
\Bigg]
=0.
\end{aligned}
\end{equation}
This is achieved at the end of \hyperref[step4]{Step 4}. Then, we pass to the limit $K\to\infty$ in \hyperref[step5]{Step 5}. To start with, we introduce some preliminaries.
\paragraph{Step 2. Block-averaged empirical measure.}
For $s\in[0,t]$ and $k\in\mathbb N_0$, define the step function
$$
p_{s,L,k}(u)
:=
\sum_{x=1}^L \mathbf 1_{I_x}(u)\,\mathbf 1_{\{\eta_x^L(s)=k\}},
\qquad u\in[0,1],
$$
which indicates whether the site corresponding to $u$ has occupation number $k$.
Then $p_{s,L,k}$ is measurable, takes values in $[0,1]$, and satisfies
$$
\sum_{k\ge0} p_{s,L,k}(u)=1
\qquad\text{for a.e. }u\in[0,1].
$$
We define the block-averaged empirical measure
\begin{equation}\label{eq:tildepi}
\widetilde \pi_s^L(du,dk)
:=
\sum_{k\ge0} p_{s,L,k}(u)\,du\,\delta_k(dk).
\end{equation}
By construction, its first marginal is the Lebesgue measure on $[0,1]$.
For a bounded measurable function $\psi:[0,1]\times\mathbb N_0\to\mathbb R$, define its
step approximation with respect to the partition $(I_x)_{x=1}^L$ by
\begin{equation}\label{def:psi_L}
\psi_L(u,k)
:=
\sum_{x=1}^L \psi(u_x,k)\,\mathbf 1_{I_x}(u).
\end{equation}
Then, for every bounded measurable $\psi$,
\begin{equation}\label{identity}
\langle \pi_s^L,\psi \rangle=\langle \widetilde\pi_s^L,\psi_L\rangle
\qquad \text{a.s.}
\end{equation}

In what follows, we will restrict to test functions
$\psi \in C_b([0,1]\times\mathbb N_0)$ with finite support in the second
variable, i.e.\ such that $\psi(\cdot,k)=0$ for $k>K$.

\paragraph{Step 2b. Control of the approximation error.}
 Using the identity \eqref{identity},
we can write
$$
\langle \widetilde\pi_s^L,\psi\rangle
-
\langle \pi_s^L,\psi\rangle
=
\langle \widetilde\pi_s^L,\psi - \psi_L\rangle \qquad \text{a.s.}
$$
Hence,
$$
\left|
\langle \widetilde\pi_s^L,\psi\rangle
-
\langle \pi_s^L,\psi\rangle
\right|
\le
\langle \widetilde\pi_s^L,\,|\psi - \psi_L|\rangle \qquad \text{a.s.}
$$
By definition of $\psi_L$, for $u\in I_x$ we have almost surely,
$$
|\psi_L(u,k)-\psi(u,k)|
=
|\psi(u_x,k)-\psi(u,k)|
\le
\sup_{\substack{u,v\in[0,1]\\ |u-v|\le 1/L}}
|\psi(u,k)-\psi(v,k)|.
$$
Since $\psi(\cdot,k)=0$ for $k>K$, taking the maximum over
$0\le k\le K$ yields
$$
\left|
\langle \widetilde\pi_s^L,\psi\rangle
-
\langle \pi_s^L,\psi\rangle
\right|
\le
\max_{0\le k\le K}
\sup_{\substack{u,v\in[0,1]\\ |u-v|\le 1/L}}
|\psi(u,k)-\psi(v,k)|,
$$
and therefore
$$
\sup_{s\in[0,t]}
\left|
\langle \widetilde\pi_s^L,\psi\rangle
-
\langle \pi_s^L,\psi\rangle
\right|
\le
\max_{0\le k\le K}\,
\sup_{\substack{u,v\in[0,1]\\ |u-v|\le 1/L}}
|\psi(u,k)-\psi(v,k)|.
$$
Since, for each $0\le k\le K$, the function $\psi(\cdot,k)$ is uniformly
continuous on $[0,1]$, the right-hand side converges to $0$ as $L\to\infty$.
Now, we already know that 
$$
\pi_s^L \Rightarrow \pi_s  \quad \text{a.s. for Lebesgue-a.e. } s\in[0,T].
$$
Moreover, for every $\psi \in C_b([0,1]\times\mathbb N_0)$ with finite support
in the second variable, we have shown that
$$
\langle \widetilde\pi_s^L,\psi\rangle
-
\langle \pi_s^L,\psi\rangle
\longrightarrow 0  \quad \text{a.s. for Lebesgue-a.e. } s\in[0,t].
$$
Therefore, for every such test function $\psi$, we can write
$$
\langle \widetilde\pi_s^L,\psi\rangle
=
\langle \pi_s^L,\psi\rangle
+
\langle \widetilde\pi_s^L - \pi_s^L,\psi\rangle \qquad \text{almost surely},
$$
and passing to the limit as $L\to\infty$, 
we get
\begin{equation}\label{eq:weakconv_prob}
\langle \widetilde\pi_s^L,\psi\rangle
\longrightarrow
\langle \pi_s,\psi\rangle  \quad \text{a.s. for Lebesgue-a.e. } s\in[0,T].
\end{equation}
for every $\psi \in C_b([0,1]\times\mathbb N_0)$ with finite support in $k$. This property will be important for the estimate of the term $I_{L,K,\varepsilon}^{(3)}(s)$ in \hyperref[step4]{Step 4} below. 
\paragraph{Step 3. Replacement of $W_L$ by $W$.}
We now start proving the limit in \eqref{eq:truncated-drift-goal}. For fixed $K$, we observe that, since both $W_L$ and $\varphi_L$ are constant on each block $I_x\times I_y$,
we have the identity
\begin{align*}
&\iint_{E^2} W_L(u,v)\,c(k,\ell)\,G_\varphi^K((u,k),(v,\ell))
\,\pi_s^L(du,dk)\,\pi_s^L(dv,d\ell)
=\\
&\hspace{2cm}\iint_{E^2} W_L(u,v)\,c(k,\ell)\,G_{\varphi_L}^K((u,k),(v,\ell))
\,\widetilde\pi_s^L(du,dk)\,\widetilde\pi_s^L(dv,d\ell).
\end{align*}
Therefore, we can go back \eqref{eq:truncated-drift-goal} with $\varphi$ replaced by $\varphi_L$ (defined as in \eqref{def:psi_L}) and consider
$$
\mathcal A_{L,K}(s)
:=
\iint_{E^2} (W_L-W)(u,v)\,c(k,\ell)\,G_{\varphi_L}^K((u,k),(v,\ell))
\,\widetilde\pi_s^L(du,dk)\,\widetilde\pi_s^L(dv,d\ell).
$$
Recalling that
\begin{equation}\label{eq:tildepi-reminder}
\widetilde \pi_s^L(du,dk)
=
\sum_{m\ge0} p_{s,L,m}(u)\,du\,\delta_m(dk),
\end{equation}
we obtain
\begin{align*}
\mathcal A_{L,K}(s)
&=
\sum_{k,\ell\ge0}
\iint_{E^2} (W_L-W)(u,v)\,c(k,\ell)\,G_{\varphi_L}^K((u,k),(v,\ell))
\,p_{s,L,k}(u)p_{s,L,\ell}(v)\,du\,dv .
\end{align*}
Since
$$
G_{\varphi_L}^K((u,k),(v,\ell))
=
G_{\varphi_L}((u,k),(v,\ell))\,\mathbf 1_{\{k\le K,\;\ell\le K\}},
$$
only the indices $k,\ell\le K$ contribute, and therefore
$$
\mathcal A_{L,K}(s)
=
\sum_{k,\ell\le K}
\iint_{E^2} (W_L-W)(u,v)\,c(k,\ell)\,G_{\varphi_L}((u,k),(v,\ell))
\,p_{s,L,k}(u)p_{s,L,\ell}(v)\,du\,dv .
$$
Now, by definition of $G_{\varphi_L}$,
\begin{align}\label{eq:G_phiL_reminder}
G_{\varphi_L}((u,k),(v,\ell))
&=
\varphi_L(u,k-1)-\varphi_L(u,k)
+
\varphi_L(v,\ell+1)-\varphi_L(v,\ell).
\end{align}
Hence
\begin{align*}
\mathcal A_{L,K}(s)
&=
\sum_{k,\ell\le K}
\iint_{E^2} (W_L-W)(u,v)\,c(k,\ell)
\Big(
\varphi_L(u,k-1)-\varphi_L(u,k)
+
\varphi_L(v,\ell+1)-\varphi_L(v,\ell)
\Big)
\\
&\hspace{5cm}\times
p_{s,L,k}(u)p_{s,L,\ell}(v)\,du\,dv .
\end{align*}
We now introduce the abbreviations
$$
a_{L,k}(u):=\varphi_L(u,k-1)-\varphi_L(u,k),
\qquad
b_{L,\ell}(v):=\varphi_L(v,\ell+1)-\varphi_L(v,\ell),
$$
so that the previous expression becomes
$$
\mathcal A_{L,K}(s)
=
\sum_{k,\ell\le K}
\iint_{E^2} (W_L-W)(u,v)\,
c(k,\ell)\big(a_{L,k}(u)+b_{L,\ell}(v)\big)
p_{s,L,k}(u)p_{s,L,\ell}(v)\,du\,dv .
$$
To estimate this quantity, we split the sum into the two corresponding terms:
\begin{equation}\label{eq:A_L}
\begin{aligned}
\mathcal A_{L,K}(s)
&=
\sum_{k,\ell\le K}
c(k,\ell)
\iint_{E^2} (W_L-W)(u,v)\,
a_{L,k}(u)\,p_{s,L,k}(u)p_{s,L,\ell}(v)\,du\,dv
\\
&\quad+
\sum_{k,\ell\le K}
c(k,\ell)
\iint_{E^2} (W_L-W)(u,v)\,
b_{L,\ell}(v)\,p_{s,L,k}(u)p_{s,L,\ell}(v)\,du\,dv .
\end{aligned}
\end{equation}
Since
$$
|a_{L,k}(u)|\le 2\|\varphi\|_\infty,
\qquad
|b_{L,\ell}(v)|\le 2\|\varphi\|_\infty,
\qquad
0\le p_{s,L,k}(u),\,p_{s,L,\ell}(v)\le 1,
$$
the functions
$$
u\mapsto \frac{a_{L,k}(u)}{2\|\varphi\|_\infty}\,p_{s,L,k}(u),
\qquad
v\mapsto p_{s,L,\ell}(v),
$$
and
$$
u\mapsto p_{s,L,k}(u),
\qquad
v\mapsto \frac{b_{L,\ell}(v)}{2\|\varphi\|_\infty}\,p_{s,L,\ell}(v),
$$
take values in $[-1,1]$. Therefore, by \eqref{eq:cut_norm}
each of the two integrals in \eqref{eq:A_L} is bounded by
$
2\|\varphi\|_\infty\,\|W_L-W\|_\square .
$
Summing the two contributions, we obtain
$$
|\mathcal A_{L,K}(s)|
\le
4\|\varphi\|_\infty
\sum_{k,\ell\le K} c(k,\ell)\,\|W_L-W\|_\square .
$$
Since $K$ is fixed and $\|W_L-W\|_\square\to0$, it follows that
\begin{equation}\label{eq:limA}
\sup_{s\in[0,t]} |\mathcal A_{L,K}(s)|
\longrightarrow 0
\qquad\text{as }L\to\infty.
\end{equation}
By adding and subtracting the term
$
\iint_{E^2} W\,c\,G_{\varphi_L}^K\,d(\widetilde\pi_s^L\otimes\widetilde\pi_s^L)
$
in \eqref{eq:truncated-drift-goal} and applying the triangle inequality, thanks to \eqref{eq:limA}, we are left with proving 
\begin{equation}\label{eq:goal_intermed}
\lim_{L\to\infty}
\mathbb E\left[
\int_0^t
\left|
\iint_{E^2} W\,c\,G_{\varphi_L}^K\,d(\widetilde\pi_s^L\otimes\widetilde\pi_s^L)
-
\iint W\,c\,G_\varphi^K\,d(\pi_s\otimes\pi_s)
\right|ds
\right]
=0.
\end{equation}

\phantomsection
\paragraph{Step 4. Proof of \eqref{eq:goal_intermed}}
\label{step4}
 Observe that the map
$
\mu \mapsto \iint_{E^2} W\,c\,G_\varphi^K\,d(\mu\otimes\mu)
$
is not continuous with respect to weak convergence of measures, since $W$
is only assumed to be bounded and measurable. To overcome this, we approximate
$W$ by a continuous function.
Let $\varepsilon>0$. Since $W\in L^\infty([0,1]^2)$ and $[0,1]^2$ has finite
measure, we have $W\in L^1([0,1]^2)$. By the density of continuous functions
in $L^1([0,1]^2)$ \cite[Proposition~7.9]{Folland}, there exists $W^\varepsilon\in C([0,1]^2)$ such that
\begin{equation}\label{eq:boundW}
\|W^\varepsilon - W\|_{L^1([0,1]^2)} < \varepsilon.
\end{equation}
Define the auxiliary functional
$$
F_{K,\varepsilon}(\mu)
:=
\iint_{E^2} W^\varepsilon(u,v)\,c(k,\ell)\,G_\varphi^K((u,k),(v,\ell))
\,\mu(du,dk)\,\mu(dv,d\ell).
$$
Since $W^\varepsilon$ is bounded and continuous, and the truncation restricts
$k,\ell$ to a finite set, the map
$
F_{K,\varepsilon}:\mathcal P([0,1]\times\mathbb N_0)\to\mathbb R
$
is bounded and continuous.
Set
$$
\mathcal B_{L,K}(s)
:=
\iint_{E^2} W(u,v)\,c(k,\ell)\,G_{\varphi_L}^K((u,k),(v,\ell))
\,\widetilde\pi_s^L(du,dk)\,\widetilde\pi_s^L(dv,d\ell),
$$
and
$$
\mathcal B_K(s)
:=
\iint_{E^2} W(u,v)\,c(k,\ell)\,G_{\varphi}^K((u,k),(v,\ell))
\,\pi_s(du,dk)\,\pi_s(dv,d\ell).
$$
We claim that
\begin{equation}\label{goal:BL1}
\int_0^t |\mathcal B_{L,K}(s)-\mathcal B_K(s)|\,ds
\overset{L\to\infty}{\longrightarrow} 0
\qquad\text{in }L^1.
\end{equation}
Now we introduce the bound and  decomposition:
$$
|\mathcal B_{L,K}(s)-\mathcal B_K(s)|
\le I_{L,K,\varepsilon}^{(1)}(s)
+ I_{L,K,\varepsilon}^{(2)}(s)
+ I_{L,K,\varepsilon}^{(3)}(s)
+ I_{L,K,\varepsilon}^{(4)}(s) \qquad \text{a.s.}
$$
where
\begin{align*}
I_{L,K,\varepsilon}^{(1)}(s)
&:=
\left|
\iint_{E^2} (W-W^\varepsilon)(u,v)\,c(k,\ell)\,G_{\varphi_L}^K((u,k),(v,\ell))
\,d(\widetilde\pi_s^L\otimes\widetilde\pi_s^L)
\right|,
\\
I_{L,K,\varepsilon}^{(2)}(s)
&:=
\left|
\iint_{E^2} W^\varepsilon(u,v)\,c(k,\ell)\,
\big(G_{\varphi_L}^K-G_\varphi^K\big)((u,k),(v,\ell))
\,d(\widetilde\pi_s^L\otimes\widetilde\pi_s^L)
\right|,
\\
I_{L,K,\varepsilon}^{(3)}(s)
&:=
\left|
F_{K,\varepsilon}(\widetilde\pi_s^L)-F_{K,\varepsilon}(\pi_s)
\right|,
\\
I_{L,K,\varepsilon}^{(4)}(s)
&:=
\left|
\iint_{E^2} (W^\varepsilon-W)(u,v)\,c(k,\ell)\,G_{\varphi}^K((u,k),(v,\ell))
\,d(\pi_s\otimes\pi_s)
\right|.
\end{align*}

We estimate the four terms separately. All the bounds below  hold almost surely.\\
\smallskip
\emph{Estimate of $I_{L,K,\varepsilon}^{(1)}(s)$.}
Recall that
$$
\widetilde\pi_s^L(du,dk)
=
\sum_{m\ge 0}p_{s,L,m}(u)\,du\,\delta_m(dk).
$$
Therefore,
$$
(\widetilde\pi_s^L\otimes \widetilde\pi_s^L)(du,dk,dv,d\ell)
=
\sum_{m,n\ge 0}
p_{s,L,m}(u)p_{s,L,n}(v)\,
du\,dv\,\delta_m(dk)\,\delta_n(d\ell).
$$
Hence
\begin{align*}
I_{L,K,\varepsilon}^{(1)}(s)
&=
\left|
\sum_{k,\ell\ge0}
\iint_{[0,1]^2}
(W-W^\varepsilon)(u,v)\,c(k,\ell)
G_{\varphi_L}^K((u,k),(v,\ell))
p_{s,L,k}(u)p_{s,L,\ell}(v)
\,du\,dv
\right| .
\end{align*}
Since $G_{\varphi_L}^K$ vanishes unless $k,\ell\le K$, this becomes
\begin{align*}
I_{L,K,\varepsilon}^{(1)}(s)
&\le
\sum_{k,\ell\le K}
\iint_{[0,1]^2}
|W-W^\varepsilon|(u,v)\,c(k,\ell)
\left|G_{\varphi_L}^K((u,k),(v,\ell))\right|
p_{s,L,k}(u)p_{s,L,\ell}(v)
\,du\,dv .
\end{align*}
Using
$$
c(k,\ell)|G_{\varphi_L}^K((u,k),(v,\ell))|
\le 4\|\varphi\|_\infty\, c(k,\ell)\,\mathbf 1_{\{k\le K,\ \ell\le K\}}
\overset{\eqref{eq:lip}}{\le} 4\|\varphi\|_\infty C_c K(1+K)\,\mathbf 1_{\{k\le K,\ \ell\le K\}},
$$
we get
\begin{align*}
I_{L,K,\varepsilon}^{(1)}(s)
&\le
4\|\varphi\|_\infty C_cK(1+K)
\iint_{[0,1]^2}
|W-W^\varepsilon|(u,v)
\sum_{k,\ell\le K}
p_{s,L,k}(u)p_{s,L,\ell}(v)
\,du\,dv .
\end{align*}
Finally,
$$
\sum_{k,\ell\le K}p_{s,L,k}(u)p_{s,L,\ell}(v)
\le
\left(\sum_{k\ge0}p_{s,L,k}(u)\right)
\left(\sum_{\ell\ge0}p_{s,L,\ell}(v)\right)
=1,
$$
and therefore
$$
I_{L,K,\varepsilon}^{(1)}(s)
\le
4\|\varphi\|_\infty\,C_cK(1+K)\,\|W-W^\varepsilon\|_{L^1([0,1]^2)}
\overset{\eqref{eq:boundW}}{\le}
C_{K,\varphi}\,\varepsilon.
$$

\smallskip
\emph{Estimate of $I_{L,K,\varepsilon}^{(2)}(s)$.}
From \eqref{eq:G_phiL_reminder}, combined with the fact that $k,\ell \leq K$ (and therefore the involved indexes can be described by some $m\in\{0,\dots, K+1\}$), we obtain:
$$
\big|G_{\varphi_L}^K - G_\varphi^K\big|
\le
C_K \max_{0\le m\le K+1}
\big|\varphi_L(\cdot,m)-\varphi(\cdot,m)\big|.
$$
Now, since $W^\varepsilon$ is bounded, we get
\begin{equation}\label{eq:I2_lim}
I_{L,K,\varepsilon}^{(2)}(s)
\le
C_{K,\varphi,\varepsilon}
\max_{0\le m\le K+1}\|\varphi_L(\cdot,m)-\varphi(\cdot,m)\|_\infty,
\end{equation}
where $C_{K,\varphi,\varepsilon}$ also incorporates the bound \eqref{eq:boundW}.
As $\varphi(\cdot,m)$ is uniformly continuous on $[0,1]$ for each fixed $m$,
the right-hand side  of \eqref{eq:I2_lim} converges to $0$ uniformly in $s$ as $L\to\infty$.

\smallskip
\emph{Estimate of $I_{L,K,\varepsilon}^{(3)}(s)$.}
Choosing in \eqref{eq:weakconv_prob} test functions of the form
$$
\psi(u,k)=f(u)\,\mathbf 1_{\{k=m\}},
\qquad f\in C([0,1]), \ m\le K+1,
$$
we obtain
$$
\int_{[0,1]} f(u)\,\widetilde\pi_s^L(du,\{m\})
\longrightarrow
\int_{[0,1]} f(u)\,\pi_s(du,\{m\}) \qquad \text{ a.s. }
$$
for every $m\le K+1$ and  for Lebesgue-a.e. $s\in[0,T]$.
For fixed $K$, the functional $F_{K,\varepsilon}$ is a finite sum, over
$0\le k,\ell\le K$, of terms of the form
$$
\iint_{[0,1]^2}
H_{k,\ell}(u,v)\,
\mu(du,\{k\})\,\mu(dv,\{\ell\}),
$$
where
$
H_{k,\ell}(u,v)
:=
W^\varepsilon(u,v)\,c(k,\ell)\,
G_\varphi((u,k),(v,\ell))
\in C_b([0,1]^2).
$
Thus $F_{K,\varepsilon}$ only tests finitely many level marginals
$\mu(\cdot,\{0\}),\ldots,\mu(\cdot,\{K\})$ against bounded continuous
functions.
Therefore
we obtain
$$
F_{K,\varepsilon}(\widetilde\pi_s^L)
\longrightarrow
F_{K,\varepsilon}(\pi_s) \qquad \text{a.s. for Lebesgue-a.e. } s\in[0,T].
$$
This gives, in particular,
$$
I_{L,K,\varepsilon}^{(3)}(s)
=
\left|
F_{K,\varepsilon}(\widetilde\pi_s^L)
-
F_{K,\varepsilon}(\pi_s)
\right|
\overset{L\to\infty}{\longrightarrow} 0  \quad \text{a.s.}
$$
for Lebesgue-a.e. $s\in[0,T]$.
Moreover, since $F_{K,\varepsilon}$ is bounded, 
$
I_{L,K,\varepsilon}^{(3)}(s)
\le 2\|F_{K,\varepsilon}\|_\infty .
$
Hence, by dominated convergence,
$$
\mathbb E\left[
\int_0^t I_{L,K,\varepsilon}^{(3)}(s)\,ds
\right]
\longrightarrow 0 \qquad \text{ as } L\to \infty.
$$
For the estimate of \(I^{(4)}_{L,K,\varepsilon}(s)\) we need first the following consideration.
\paragraph{Spatial marginal of the limit measure.}
We claim that every limit point \((\pi_t)_{t\in[0,T]}\) has
Lebesgue first marginal. Let
\[
\Gamma:\mathcal P([0,1]\times\mathbb N_0)\to\mathcal P([0,1]),
\qquad
\Gamma(\mu)(A):=\mu(A\times\mathbb N_0),
\]
be the projection onto the first coordinate. The map \(\Gamma\) is continuous
with respect to weak convergence. Therefore, by the continuous mapping
theorem,
\[
\Gamma(\pi^L)\Rightarrow \Gamma(\pi)
\qquad
\text{in }D([0,T],\mathcal P([0,1])).
\]
On the other hand, for every \(L\) and every \(t\in[0,T]\),
\[
\Gamma(\pi_t^L)
=
\frac1L\sum_{x=1}^L\delta_{x/L}
=:\lambda_L.
\]
Thus \(\Gamma(\pi^L)\) is the constant path \(t\mapsto\lambda_L\). Since
\[
\lambda_L\Rightarrow du
\qquad\text{in }\mathcal P([0,1]),
\]
we  have
$
\Gamma(\pi_t)=du$
for all  $t\in[0,T].
$
Equivalently,
since \(\Gamma(\pi_t)\) is the first marginal of \(\pi_t\), we get
\begin{equation}\label{eq:leb_marg}
\pi_t(du,\mathbb N_0)=du
\qquad\text{for all }t\in[0,T].
\end{equation}

\textit{Estimate of \(I^{(4)}_{L,K,\varepsilon}(s)\).}
Using the definition of \(G_\phi^K\) and the growth assumption on the rates,
we have
\[
c(k,\ell)|G_\phi^K((u,k),(v,\ell))|
\le
4\|\phi\|_\infty C_c K(1+K)
\mathbf 1_{\{k\le K,\ \ell\le K\}}.
\]
Hence
\[
\begin{aligned}
I^{(4)}_{L,K,\varepsilon}(s)
&\le
C_{K,\phi}
\sum_{k,\ell\le K}
\iint_{[0,1]^2}
|W^\varepsilon-W|(u,v)\,
\pi_s(du,\{k\})\pi_s(dv,\{\ell\})
\\
&\le
C_{K,\phi}
\iint_{[0,1]^2}
|W^\varepsilon-W|(u,v)\,
\pi_s(du,\mathbb N_0)\pi_s(dv,\mathbb N_0).
\end{aligned}
\]
From \eqref{eq:leb_marg}
we obtain
\[
\begin{aligned}
I^{(4)}_{L,K,\varepsilon}(s)
&\le
C_{K,\phi}
\iint_{[0,1]^2}
|W^\varepsilon-W|(u,v)\,du\,dv
=
C_{K,\phi}
\|W^\varepsilon-W\|_{L^1([0,1]^2)}\\
&\le
C_{K,\phi}\varepsilon .
\end{aligned}
\]
This concludes the proof of \eqref{eq:goal_intermed}.

\phantomsection
\paragraph{Step 5. Limit of the martingale formulation.}\label{step5} 

We now collect all the contributions. To ease notation, let us denote
$$
\mathcal D_L(s)
:=
\iint_{E^2} W_L(u,v)\,c(k,\ell)\,G_\varphi((u,k),(v,\ell))\,
\pi_s^L(du,dk)\pi_s^L(dv,d\ell),
$$
$$
\mathcal D_L^K(s)
:=
\iint_{E^2} W_L(u,v)\,c(k,\ell)\,G_\varphi^K((u,k),(v,\ell))\,
\pi_s^L(du,dk)\pi_s^L(dv,d\ell),
$$
and similarly
$$
\mathcal D(s)
:=
\iint_{E^2} W(u,v)\,c(k,\ell)\,G_\varphi((u,k),(v,\ell))\,
\pi_s(du,dk)\pi_s(dv,d\ell),
$$
$$
\mathcal D^K(s)
:=
\iint_{E^2} W(u,v)\,c(k,\ell)\,G_\varphi^K((u,k),(v,\ell))\,
\pi_s(du,dk)\pi_s(dv,d\ell).
$$
Taking expectations in the above inequality, we obtain, for every fixed \(K\),
\[
\begin{aligned}
\limsup_{L\to\infty}
\mathbb E\left[
\left|
\int_0^t \mathcal D_L(s)\,ds
-
\int_0^t \mathcal D(s)\,ds
\right|
\right]
&\le
\limsup_{L\to\infty}
\mathbb E\left[
\int_0^t
\left|\mathcal D_L^K(s)-\mathcal D^K(s)\right|\,ds
\right]
\\
&\quad+
\limsup_{L\to\infty}
\mathbb E\left[
\int_0^t
\left|\mathcal D_L(s)-\mathcal D_L^K(s)\right|\,ds
\right]
\\
&\quad+
\mathbb E\left[
\int_0^t
\left|\mathcal D(s)-\mathcal D^K(s)\right|\,ds
\right].
\end{aligned}
\]
The first term on the right-hand side is zero for every fixed \(K\) (see
\eqref{eq:limA} and \eqref{eq:goal_intermed}). Hence
\[
\begin{aligned}
\limsup_{L\to\infty}
\mathbb E\left[
\left|
\int_0^t \mathcal D_L(s)\,ds
-
\int_0^t \mathcal D(s)\,ds
\right|
\right]
&\le
\sup_L
\mathbb E\left[
\int_0^t
\left|\mathcal D_L(s)-\mathcal D_L^K(s)\right|\,ds
\right]
\\
&\quad+
\mathbb E\left[
\int_0^t
\left|\mathcal D(s)-\mathcal D^K(s)\right|\,ds
\right].
\end{aligned}
\]
Letting now \(K\to\infty\), and using
\eqref{eq:tail-prelimit}--\eqref{eq:tail-limit}, we obtain
\[
\limsup_{L\to\infty}
\mathbb E\left[
\left|
\int_0^t \mathcal D_L(s)\,ds
-
\int_0^t \mathcal D(s)\,ds
\right|
\right]
=0.
\]
Therefore,
\[
\mathbb E\left[
\left|
\int_0^t \mathcal D_L(s)\,ds
-
\int_0^t \mathcal D(s)\,ds
\right|
\right]
\longrightarrow 0,
\qquad L\to\infty.
\]
This proves \eqref{eq:L1_goal}. In particular, since convergence in \(L^1\)
implies convergence in probability, we obtain
\begin{equation}
\begin{aligned}\label{eq:full-drift-convergence}
&\int_0^t
\iint W_L(u,v)c(k,\ell)G_\varphi((u,k),(v,\ell))
\,\pi_s^L(du,dk)\pi_s^L(dv,d\ell)\,ds
\\
&\hspace{2cm}\xrightarrow[L\to\infty]{\mathbb P}
\int_0^t
\iint W(u,v)c(k,\ell)G_\varphi((u,k),(v,\ell))
\,\pi_s(du,dk)\pi_s(dv,d\ell)\,ds .
\end{aligned}
\end{equation}
We now pass to the limit in the martingale formulation. For every $t\in[0,T]$,
$$
\langle \pi_t^L,\varphi\rangle
=
\langle \pi_0^L,\varphi\rangle
+
\int_0^t \mathcal D_L(s)\,ds
+
M_\Phi^L(t).
$$
Sending $L\to \infty$ we have, in probability:
\begin{itemize}
\item $\langle \pi_t^L,\varphi\rangle \to \langle \pi_t,\varphi\rangle$
by \eqref{eq:weak_conv_pi}  for Lebesgue-a.e. $t\in[0,T]$;
\item $\langle \pi_0^L,\varphi\rangle \to \langle \pi_0,\varphi\rangle$ by the
assumption \eqref{init_conv} on the initial condition;
\item $M_\Phi^L(t)\to 0$, uniformly on compact time intervals;
\item the drift term converges by \eqref{eq:full-drift-convergence}.
\end{itemize}

Therefore every limit point $\pi$ satisfies the weak formulation
\begin{equation}\label{eq:weak}
\langle \pi_t,\varphi\rangle
=
\langle \pi_0,\varphi\rangle
+
\int_0^t
\iint W(u,v)\,c(k,\ell)\,G_\varphi((u,k),(v,\ell))\,
\pi_s(du,dk)\pi_s(dv,d\ell)\,ds
\end{equation}
for every $\varphi\in C_b([0,1]\times\mathbb N_0)$. 

\subsection{Uniqueness of weak solutions} \label{subsec:uniqueness}

\begin{lemma}\label{lem:moment_limit}
Let $(\mu_t)_{t\ge 0}$ be a weak solution of \eqref{eq:weak_limit} such that for every $T>0,$
\begin{equation}\label{bound:m1}
m_1(t):=\int_{[0,1]\times \N_0} k\,\mu_t(du,dk)<C_T \text{ for all } t\in [0,T]
\end{equation}
for some positive constant $C_T$, and \[
m_2(0):=\int_{[0,1]\times \N_0} k^2\,\mu_0(du,dk)<\infty.
\]
Then there exists a constant $C>0$ such that
\begin{equation}\label{eq:moment_limit_bound}
m_2(t):=\int_{[0,1]\times \N_0} k^2\,\mu_t(du,dk)
\le \bigl(1+m_2(0)\bigr)e^{Ct}
\qquad \text{for all } t\in [0,T].
\end{equation}
In particular, for every $T>0$,
\[
\sup_{t\in[0,T]} \int_{[0,1]\times \N_0} k^2\,\mu_t(du,dk)<\infty.
\]
\end{lemma}

\begin{proof}
For $M\in \N$, define the bounded test function
\[
\phi_M(u,k):=k^2\wedge M, \qquad (u,k)\in [0,1]\times \N_0.
\]
Since $(\mu_t)_{t\ge 0}$ is a weak solution of \eqref{eq:weak_limit}, for every
$t\ge 0$ we have
\begin{align}
\langle \mu_t,\phi_M\rangle
&=
\langle \mu_0,\phi_M\rangle
+
\int_0^t
\iint
W(u,v)c(k,\ell)\,
G_{\phi_M}((u,k),(v,\ell))\,
\mu_s(du,dk)\mu_s(dv,d\ell)\,ds,
\label{eq:weak_phiM}
\end{align}
where
\[
G_{\phi_M}((u,k),(v,\ell))
=
\phi_M(u,k-1)-\phi_M(u,k)+\phi_M(v,\ell+1)-\phi_M(v,\ell).
\]
Since \(\phi_M\) is nondecreasing in the second variable,
\[
\phi_M(u,k-1)-\phi_M(u,k)\le 0.
\]
Moreover, the map \(r\mapsto r\wedge M\) is nondecreasing and \(1\)-Lipschitz, hence
\[
\phi_M(v,\ell+1)-\phi_M(v,\ell)
\le (\ell+1)^2-\ell^2=2\ell+1.
\]
Therefore,
\[
G_{\phi_M}((u,k),(v,\ell))\le 2\ell+1.
\]
Using the growth assumption \eqref{eq:lip} and the boundedness of $W$, we obtain
\begin{align*}
&W(u,v)c(k,\ell)G_{\phi_M}((u,k),(v,\ell))
\\
&\qquad
\le 2C_WC_c\,k(1+\ell)^2
\le 2C_WC_c\,k(1+2\ell+\ell^2).
\end{align*}
Integrating with respect to $\mu_s(du,dk)\mu_s(dv,d\ell)$ and using that the
first moment is bounded on $[0,T]$, we get
\begin{align*}
\frac{d}{dt}\langle \mu_t,\phi_M\rangle
&\le
2C_WC_c
\left(\int_{[0,1]\times \N_0} k\,\mu_t(du,dk)\right)
\left(\int_{[0,1]\times \N_0} (1+2\ell+\ell^2)\,\mu_t(dv,d\ell)\right)
\\
&\le
2C_WC_c\, C_T
\left(1+2 C_T+\int_{[0,1]\times \N_0} \ell^2\,\mu_t(dv,d\ell)\right)\\
&\leq \bar C \bigl(1+m_2(t)\bigr),
\end{align*}
for a constant $\bar C>0$ depending only on $C_W$, $C_c$ and $C_T$.
Since \(\phi_M(k)\uparrow k^2\), as $M\to\infty$, by the monotone convergence theorem
\[
\langle\mu_t,\phi_M\rangle \uparrow m_2(t),
\qquad
\langle\mu_0,\phi_M\rangle \uparrow m_2(0).
\]
Moreover,
\[
\langle \mu_t,\phi_M\rangle
\le
\langle \mu_0,\phi_M\rangle
+
\bar C\int_0^t (1+m_2(s))\,ds.
\]
Passing to the limit \(M\to\infty\) yields
\[
m_2(t)
\le
m_2(0)
+
\bar C\int_0^t (1+m_2(s))\,ds.
\]
By Gronwall's lemma, for every \(t\in[0,T]\),
\[
m_2(t)\le \bigl(1+m_2(0)\bigr)e^{\bar C t}.
\]
Renaming the constant concludes the proof.
\end{proof}


\begin{proposition}[Stability and uniqueness]\label{prop:uniqueness}
Let $(\mu_t)_{t\in[0,T]}$ and $(\nu_t)_{t\in[0,T]}$ be two weak solutions of
\eqref{eq:weak_limit} satisfying the assumptions of Lemma \ref{lem:moment_limit}.
For $k\ge0$, write
\[
\mu_t^k(du):=\mu_t(du,\{k\}),\qquad
\nu_t^k(du):=\nu_t(du,\{k\}),
\]
and set
\[
\Delta_t^k:=\mu_t^k-\nu_t^k .
\]
Define
\[
\Theta(t):=\sum_{k\ge0}(1+k)|\Delta_t^k|([0,1]).
\]
Then there exists a constant $C_T>0$ such that
\[
\Theta(t)\le e^{C_Tt}\Theta(0),
\qquad t\in[0,T].
\]
In particular, if $\mu_0=\nu_0$, then $\mu_t=\nu_t$ for every $t\in[0,T]$.
\end{proposition}


\begin{proof}

We define
\begin{equation}\label{eq:theta_M}
\Theta_M(t):=\sum_{k=0}^M(1+k)\|\Delta_t^k\|_{\rm TV}.
\end{equation}

First observe that the \(u\)-marginal is conserved. Indeed, if
\(\varphi(u,k)=\psi(u)\), then
$
G_\varphi((u,k),(v,\ell))=0.
$
Hence the weak formulation yields
\[
\int_{[0,1]\times\N_0}\psi(u)\,\mu_t(du,dk)
=
\int_{[0,1]\times\N_0}\psi(u)\,\mu_0(du,dk)
\]
for every bounded continuous function \(\psi\). 
Therefore
\begin{equation}\label{eq:cons_marg}
\sum_{k\ge0}\mu_t^k=\sum_{k\ge0}\mu_0^k,
\qquad
\sum_{k\ge0}\nu_t^k=\sum_{k\ge0}\nu_0^k.
\end{equation}
Consider the fixed dominating measure
\[
\zeta(du):=\sum_{k\ge0}\mu_0^k(du)
        +\sum_{k\ge0}\nu_0^k(du).
\]
From \eqref{eq:cons_marg},
\[
\mu_t^k\le \sum_{j\ge0}\mu_t^j=\sum_{j\ge0}\mu_0^j\le \zeta,
\qquad
\nu_t^k\le \sum_{j\ge0}\nu_t^j=\sum_{j\ge0}\nu_0^j\le \zeta.
\]
Hence, for each \(k\),
\[
\mu_t^k\ll \zeta,\qquad \nu_t^k\ll\zeta,\qquad \Delta_t^k\ll\zeta.
\]
Write
\[
\mu_t^k=p_t^k\zeta,\qquad
\nu_t^k=q_t^k\zeta,\qquad
\Delta_t^k=f_t^k\zeta,
\]
where \(f_t^k=p_t^k-q_t^k\).
Fix \(k\in\N_0\). By testing the weak formulation against functions of the form
\(\varphi(u,m)=\psi(u)\mathbf 1_{\{m=k\}}\), one obtains that the
\(k\)-th marginal satisfies the  equation
\[
\partial_t \mu_t^k
=
A_{\mu,t}(\cdot,k+1)\mu_t^{k+1}
-
A_{\mu,t}(\cdot,k)\mu_t^k
+
\mathbf 1_{\{k\ge1\}}
B_{\mu,t}(\cdot,k-1)\mu_t^{k-1}
-
B_{\mu,t}(\cdot,k)\mu_t^k,
\]
where
\[
A_{\mu,t}(u,k)
:=
\sum_{\ell\ge0}
\int_{[0,1]}
W(u,v)c(k,\ell)\,\mu_t^\ell(dv),
\]
and
\[
B_{\mu,t}(u,k)
:=
\sum_{\ell\ge0}
\int_{[0,1]}
W(v,u)c(\ell,k)\,\mu_t^\ell(dv).
\]
Here \(A_{\mu,t}(u,k)\) is the rate at which a site with occupancy \(k\)
loses one particle, while \(B_{\mu,t}(u,k)\) is the rate at which it gains
one particle. Similarly one can define \(A_{\nu,t}\) and \(B_{\nu,t}\).
Since \(\mu_t^k=p_t^k\zeta\) and \(\nu_t^k=q_t^k\zeta\), the previous
equation can be written in terms of densities as
\[
\partial_t p_t^k
=
A_{\mu,t}(\cdot,k+1)p_t^{k+1}
-
A_{\mu,t}(\cdot,k)p_t^k
+
\mathbf 1_{\{k\ge1\}}
B_{\mu,t}(\cdot,k-1)p_t^{k-1}
-
B_{\mu,t}(\cdot,k)p_t^k,
\]
and analogously
\[
\partial_t q_t^k
=
A_{\nu,t}(\cdot,k+1)q_t^{k+1}
-
A_{\nu,t}(\cdot,k)q_t^k
+
\mathbf 1_{\{k\ge1\}}
B_{\nu,t}(\cdot,k-1)q_t^{k-1}
-
B_{\nu,t}(\cdot,k)q_t^k.
\]
Subtracting the two equations and recalling that
\(f_t^k=p_t^k-q_t^k\), we obtain
\[
\partial_t f_t^k=h_t^k
\qquad\text{in }L^1(\zeta),
\]
where
\begin{equation}\label{eq:h_t}
\begin{aligned}
h_t^k(u)
={}&
A_{\mu,t}(u,k+1)p_t^{k+1}(u)
-
A_{\mu,t}(u,k)p_t^k(u)
\\
&+
\mathbf 1_{\{k\ge1\}}
B_{\mu,t}(u,k-1)p_t^{k-1}(u)
-
B_{\mu,t}(u,k)p_t^k(u)
\\
&-
A_{\nu,t}(u,k+1)q_t^{k+1}(u)
+
A_{\nu,t}(u,k)q_t^k(u)
\\
&-
\mathbf 1_{\{k\ge1\}}
B_{\nu,t}(u,k-1)q_t^{k-1}(u)
+
B_{\nu,t}(u,k)q_t^k(u).
\end{aligned}
\end{equation}
Using the moment bound \eqref{bound:m1} together with the growth assumption \eqref{eq:lip} on \(c\), each term in \(h_t^k\) is therefore
integrable in \(L^1(\zeta)\), and for every fixed \(k\) there exists a constant
\(C_{k,T}<\infty\) such that
\[
\|h_t^k\|_{L^1(\zeta)}\le C_{k,T}
\qquad\text{for a.e. }t\in[0,T].
\]
Hence \(h^k\in L^1([0,T];L^1(\zeta))\). Therefore,
$
f^k \in W^{1,1}\bigl([0,T];L^1(\zeta)\bigr),
$
with weak derivative \(h^k\). Consequently, the map
$$
t \longmapsto \|f_t^k\|_{L^1(\zeta)}
$$
is absolutely continuous on \([0,T]\) and the \(L^1\)-chain rule applies.
Consequently, for every \(M\), the map
\[
t\mapsto \Theta_M(t)
=
\sum_{k=0}^M(1+k)\int_{[0,1]} |f_t^k(u)|\,\zeta(du)
\]
is absolutely continuous, and for a.e. \(t\),
\[
\frac d{dt}\Theta_M(t)
=
\sum_{k=0}^M(1+k)
\int_{[0,1]} \operatorname{sgn}(f_t^k(u))h_t^k(u)\,\zeta(du).
\]
Set
\[
s_k(t,u):=(1+k)\mathbf 1_{\{k\le M\}}
\operatorname{sgn}(f_t^k(u)),
\qquad
s_{-1}(t,u):=0.
\]
Then
\[
\frac d{dt}\Theta_M(t)
=
\sum_{k\ge0}\int_{[0,1]} s_k(t,u)h_t^k(u)\,\zeta(du).
\]
Substituting \eqref{eq:h_t} in the display above, and re-indexing the gain terms according to \(k\mapsto k+1\) in the
\(A\)-contribution and \(k\mapsto k-1\) in the \(B\)-contribution, yields
\[
\begin{aligned}
\frac d{dt}\Theta_M(t)
={}&
\sum_{k,\ell\ge0}
\iint_{[0,1]^2}
\Big[
W(u,v)c(k,\ell)(s_{k-1}(t,u)-s_k(t,u))
\\
&\qquad\qquad\qquad
+
W(v,u)c(\ell,k)(s_{k+1}(t,u)-s_k(t,u))
\Big]
\\
&\qquad\qquad\qquad
\times
\big[
\mu_t^k(du)\mu_t^\ell(dv)
-
\nu_t^k(du)\nu_t^\ell(dv)
\big].
\end{aligned}
\]
Using 
\[
\mu_t^k\otimes\mu_t^\ell-\nu_t^k\otimes\nu_t^\ell
=
\Delta_t^k\otimes\mu_t^\ell
+
\nu_t^k\otimes\Delta_t^\ell,
\]
we write
\begin{equation}\label{eq:Decomp_theta}
\frac d{dt}\Theta_M(t)=I_M(t)+J_M(t),
\end{equation}
with 
\[
\begin{aligned}
I_M(t)
:={}&
\sum_{k,\ell\ge0}
\iint_{[0,1]^2}
\Big[
W(u,v)c(k,\ell)(s_{k-1}(t,u)-s_k(t,u))
\\
&\qquad\qquad\qquad
+
W(v,u)c(\ell,k)(s_{k+1}(t,u)-s_k(t,u))
\Big]
\,
\Delta_t^k(du)\,
\mu_t^\ell(dv),
\\[0.5em]
J_M(t)
:={}&
\sum_{k,\ell\ge0}
\iint_{[0,1]^2}
\Big[
W(u,v)c(k,\ell)(s_{k-1}(t,u)-s_k(t,u))
\\
&\qquad\qquad\qquad
+
W(v,u)c(\ell,k)(s_{k+1}(t,u)-s_k(t,u))
\Big]
\,
\nu_t^k(du)\,
\Delta_t^\ell(dv).
\end{aligned}
\]
We now want to bound them. We first consider the indices \(0\le k\le M\). For the \(I_M\) term, we use the special sign structure. Since
\[
\Delta_t^k=f_t^k\zeta,
\qquad
\operatorname{sgn}(f_t^k)\Delta_t^k=|\Delta_t^k|,
\]
one has
\begin{equation}\label{eq:first_Dis_s}
(s_{k-1}-s_k)\Delta_t^k\le 0,
\end{equation}
and
\begin{equation}\label{eq:sec_Dis_s}
(s_{k+1}-s_k)\Delta_t^k\le |\Delta_t^k|.
\end{equation}
For the first inequality (the second one is treated similarly), we used
\[
\begin{aligned}
(s_{k-1}(t,u)-s_k(t,u))\,\Delta_t^k(du)
&=
(s_{k-1}(t,u)-s_k(t,u))f_t^k(u)\,\zeta(du)
\\
&\le
\bigl(k-(1+k)\bigr)|f_t^k(u)|\,\zeta(du)
\\
&=
-|f_t^k(u)|\,\zeta(du)
\le 0.
\end{aligned}
\]
For the contribution with \(0\le k\le M\), using
\[
(s_{k-1}-s_k)\Delta_t^k\le 0,
\qquad
(s_{k+1}-s_k)\Delta_t^k\le |\Delta_t^k|,
\]
together with \(W\ge0\), \(c\ge0\), and \(\|W\|_\infty<\infty\), we get\footnote{The constant \(C\) may change from line to line.}
\begin{equation}\label{eq:contrib1}
\begin{aligned}
&\sum_{k=0}^M\sum_{\ell\ge0}
\iint_{[0,1]^2}
\Big[
W(u,v)c(k,\ell)(s_{k-1}(t,u)-s_k(t,u))
\\
&\qquad\qquad\qquad\qquad
+
W(v,u)c(\ell,k)(s_{k+1}(t,u)-s_k(t,u))
\Big]
\Delta_t^k(du)\mu_t^\ell(dv)
\\
&\qquad \overset{\eqref{eq:first_Dis_s},\eqref{eq:sec_Dis_s}}{\le}
C
\sum_{k=0}^M\sum_{\ell\ge0}
c(\ell,k)\,
\|\Delta_t^k\|_{\rm TV}\,
\mu_t^\ell([0,1])
\overset{\eqref{eq:lip}}{\le}
C
\sum_{k=0}^M\sum_{\ell\ge0}
\ell(1+k)\,
\|\Delta_t^k\|_{\rm TV}\,
\mu_t^\ell([0,1])
\\
&\qquad=
C
\left(\sum_{\ell\ge0}\ell\,\mu_t^\ell([0,1])\right)
\left(\sum_{k=0}^M(1+k)\|\Delta_t^k\|_{\rm TV}\right)\overset{\eqref{bound:m1}}{\le}
C^{(1)}_T\Theta(t).
\end{aligned}
\end{equation}
For \(k\ge M+2\), both differences \(s_{k-1}-s_k\) and
\(s_{k+1}-s_k\) vanish. Hence the only contribution not covered by the
previous sign estimates comes from the boundary index \(k=M+1\). 
Since \(s_{M+1}=0\), we have \(s_M-s_{M+1}=s_M\), and
$
|s_M(t,u)|\le M+1$.
For this isolated term we produce the bound
\begin{equation}\label{eq:contrib2}
\begin{aligned}
&\sum_{\ell\ge0}
\iint_{[0,1]^2}
W(u,v)c(M+1,\ell)
\bigl(s_M-s_{M+1}\bigr)
\Delta_t^{M+1}(du)\mu_t^\ell(dv)\\
&\le
C_W(M+1)
\sum_{\ell\ge0}
\iint_{[0,1]^2}
c(M+1,\ell)\,
|\Delta_t^{M+1}|(du)\mu_t^\ell(dv)
\\
&\le
C_WC_c(M+1)^2
\sum_{\ell\ge0}
(1+\ell)
|\Delta_t^{M+1}|([0,1])\mu_t^\ell([0,1])
\\
&=
C_WC_c(M+1)^2
\|\Delta_t^{M+1}\|_{\rm TV}
\bigl(1+m_1(t)\bigr)
\\
&\le
C^{(2)}_T(M+1)^2\|\Delta_t^{M+1}\|_{\rm TV}.
\end{aligned}
\end{equation}
All together \eqref{eq:contrib1},\eqref{eq:contrib2} give
\begin{equation}\label{eq:estI}
I_M(t)
\le
C^{(1)}_T\Theta(t)
+
C^{(2)}_T(M+1)^2\|\Delta_t^{M+1}\|_{\rm TV}.
\end{equation}
We now move to the term \(J_M\), and we estimate it in absolute value. Since
\[
|s_{k-1}-s_k|+|s_{k+1}-s_k|
\le C(1+k),
\]
we get
\[
\begin{aligned}
|J_M(t)|
&\le
C
\sum_{k,\ell\ge0}
(1+k)
\big(c(k,\ell)+c(\ell,k)\big)
\nu_t^k([0,1])\|\Delta_t^\ell\|_{\rm TV}.
\end{aligned}
\]
Using
\[
c(k,\ell)\le C_c k(1+\ell),
\qquad
c(\ell,k)\le C_c\ell(1+k),
\]
we find
\[
|J_M(t)|
\le
C
\left(
\sum_{k\ge0}(1+k)^2\nu_t^k([0,1])
\right)
\left(
\sum_{\ell\ge0}(1+\ell)\|\Delta_t^\ell\|_{\rm TV}
\right).
\]
Lemma~\ref{lem:moment_limit} provides a uniform bound on $[0,T]$ for the second moment of \(\nu_t\), hence
\begin{equation}\label{eq:estJ}
|J_M(t)|\le C^{(3)}_T\Theta(t).
\end{equation}
Combining \eqref{eq:estI}, \eqref{eq:estJ} and renaming constants, we obtain from \eqref{eq:Decomp_theta} that for a.e. \(t\in[0,T]\),
\[
\frac d{dt}\Theta_M(t)
\le
C_T\Theta(t)
+
C_T(M+1)^2\|\Delta_t^{M+1}\|_{\rm TV}.
\]
Therefore
\begin{equation}\label{eq:boundTheta}
\Theta_M(t)
\le \Theta_M(0)+
C_T\int_0^t \Theta(s)\,ds
+
C_T\int_0^t
(M+1)^2\|\Delta_s^{M+1}\|_{\rm TV}\,ds.
\end{equation}
We now let \(M\to\infty\). Since
\[
\|\Delta_s^{M+1}\|_{\rm TV}
\le
\mu_s^{M+1}([0,1])+\nu_s^{M+1}([0,1]),
\]
and
\[
(M+1)^2\mu_s^{M+1}([0,1])
\le
\sum_{k\ge M+1}k^2\mu_s^k([0,1]),
\]
with an analogous estimate for \(\nu_s\), we obtain
\[
\begin{aligned}
(M+1)^2\|\Delta_s^{M+1}\|_{\rm TV}
&\le
(M+1)^2\mu_s^{M+1}([0,1])
+
(M+1)^2\nu_s^{M+1}([0,1])
\\
&\le
\sum_{k\ge M+1}k^2\mu_s^k([0,1])
+
\sum_{k\ge M+1}k^2\nu_s^k([0,1]).
\end{aligned}
\]
From the second moments bounds of Lemma~\ref{lem:moment_limit} we deduce that as $M\to\infty$
\[
\sum_{k\ge M+1}k^2\mu_s^k([0,1])\downarrow 0,
\qquad
\sum_{k\ge M+1}k^2\nu_s^k([0,1])\downarrow 0
\]
for every \(s\), and both tails are bounded by the uniformly integrable function $m_2(s)$, dominated convergence gives
\[
\int_0^t
(M+1)^2\|\Delta_s^{M+1}\|_{\rm TV}\,ds
\overset{M\to\infty}{\longrightarrow} 0.
\]
Moreover, \(\Theta_M(t)\uparrow\Theta(t)\). Hence passing to the limit $M\to\infty$ in \eqref{eq:boundTheta} we obtain
\[
\Theta(t)
\le \Theta(0)+
C_T\int_0^t \Theta(s)\,ds.
\]
By Gronwall's lemma,
\[
\Theta(t)\leq \Theta(0) e^{C_Tt}
\qquad\text{for all }t\in[0,T].
\]
In particular, if \(\mu_0=\nu_0\), then \(\Delta_0^k=0\) for every
\(k\ge0\), and hence \(\Theta(0)=0\). Therefore
$
\Theta(t)=0$,
$t\in[0,T].$
Since
$
\Theta(t)=\sum_{k\ge0}(1+k)\|\Delta_t^k\|_{\rm TV},
$
and every term in the sum is nonnegative, it follows that
\[
\|\Delta_t^k\|_{\rm TV}=0
\qquad\text{for every }k\ge0.
\]
Thus \(\Delta_t^k=0\), i.e. \(\mu_t^k=\nu_t^k\), for every \(k\ge0\).
Consequently,
\[
\mu_t=\nu_t
\qquad\text{for every }t\in[0,T],
\]
which proves uniqueness.
\end{proof}

\subsection{Conclusion of the proof of Theorem~\ref{mainthm}}
\begin{proof}[Proof of Thm.~\ref{mainthm}]
By tightness, provided by Prop.~\ref{tightness}, every subsequence admits a convergent subsequence. The
identification step provided in Subsec.~\ref{subsec:id_limit} shows that any limit point is a weak solution of
\eqref{eq:weak_limit}. By Proposition~\ref{prop:uniqueness}, such a
solution is unique. Therefore all limit points coincide, and the whole
sequence converges. This completes the proof.
\end{proof}
\bibliographystyle{amsplain}
\bibliography{ref_new}

@book {pollard,
    AUTHOR = {Pollard, David},
     TITLE = {Convergence of stochastic processes},
    SERIES = {Springer Series in Statistics},
 PUBLISHER = {Springer-Verlag, New York},
      YEAR = {1984},
     PAGES = {xiv+215},
      ISBN = {0-387-90990-7},
   MRCLASS = {60F05 (60B10)},
  MRNUMBER = {762984},
MRREVIEWER = {R. M. Dudley},
       DOI = {10.1007/978-1-4612-5254-2},
       URL = {https://doi-org.virtual.anu.edu.au/10.1007/978-1-4612-5254-2},
}

@article{AIM1,
  author  = {Andreis, L. and Iyer, T. and Magnanini, E.},
  year={2026},
  title   = {Gelation in cluster coagulation processes},
  journal = {Ann. Inst. H. Poincar{\'e} Probab. Statist.},
  note    = {To appear. arXiv:2308.10232. DOI: 10.1214/25-AIHP1581}
}

@article {AIM2,
    AUTHOR = {Andreis, Luisa and Iyer, Tejas and Magnanini, Elena},
     TITLE = {Convergence of cluster coagulation dynamics},
   JOURNAL = {Electron. J. Probab.},
  FJOURNAL = {Electronic Journal of Probability},
    VOLUME = {31},
      YEAR = {2026},
      pages={1--29},
      ISSN = {1083-6489},
   MRCLASS = {99-06},
  MRNUMBER = {5073009},
       DOI = {10.1214/26-ejp1539},
       URL = {https://doi.org/10.1214/26-ejp1539},
}

@book {Folland,
    AUTHOR = {Folland, Gerald B.},
     TITLE = {Real analysis},
    SERIES = {Pure and Applied Mathematics (New York)},
   EDITION = {Second},
      NOTE = {Modern techniques and their applications,
              A Wiley-Interscience Publication},
 PUBLISHER = {John Wiley \& Sons, Inc., New York},
      YEAR = {1999},
     PAGES = {xvi+386},
      ISBN = {0-471-31716-0},
   MRCLASS = {00A05 (26-01 28-01 46-01)},
  MRNUMBER = {1681462},
}

@Article{armendariz2013zero,
  Title                    = {Zero-range condensation at criticality},
  Author                   = {Armend{\'a}riz, In{\'e}s and Grosskinsky, Stefan and Loulakis, Michail},
  Journal                  = {Stochastic Processes and their Applications},
  Year                     = {2013},
  Number                   = {9},
  Pages                    = {3466--3496},
  Volume                   = {123},
  Publisher                = {Elsevier}
}

@Book{billingsley2013convergence,
  Title                    = {Convergence of probability measures},
  Author                   = {Billingsley, Patrick},
  Publisher                = {John Wiley \& Sons},
  Year                     = {2013}
}

@Article{chleboun2014condensation,
  Title                    = {Condensation in stochastic particle systems with stationary product measures},
  Author                   = {Chleboun, Paul and Grosskinsky, Stefan},
  Journal                  = {Journal of Statistical Physics},
  Year                     = {2014},
  Number                   = {1-2},
  Pages                    = {432--465},
  Volume                   = {154},
  Publisher                = {Springer}
}

@Article{cocozza1985processus,
  Title                    = {Processus des misanthropes},
  Author                   = {Cocozza-Thivent, Christiane},
  Journal                  = {Zeitschrift f{\"u}r Wahrscheinlichkeitstheorie und verwandte Gebiete},
  Year                     = {1985},
  Number                   = {4},
  Pages                    = {509--523},
  Volume                   = {70},
  Publisher                = {Springer}
}

@Book{ethier2009markov,
  Title                    = {Markov processes: characterization and convergence},
  Author                   = {Ethier, Stewart N and 
Kurtz, Thomas G},
  Publisher                = {John Wiley \& Sons},
  Year                     = {2009},
  Volume                   = {282}
}

@Article{evans2014condensation,
  Title                    = {Condensation in stochastic mass transport models: beyond the zero-range process},
  Author                   = {Evans, Martin R and Waclaw, Bart{\l}omiej},
  Journal                  = {Journal of Physics A: Mathematical and Theoretical},
  Year                     = {2014},
  Number                   = {9},
  Pages                    = {095001},
  Volume                   = {47},
  Publisher                = {IOP Publishing}
}

@Article{godreche2003dynamics,
  Title                    = {Dynamics of condensation in zero-range processes},
  Author                   = {Godr{\`e}che, Claude},
  Journal                  = {Journal of Physics A: Mathematical and General},
  Year                     = {2003},
  Number                   = {23},
  Pages                    = {6313},
  Volume                   = {36},
  Publisher                = {IOP Publishing}
}

@Article{godreche2016coarsening,
  Title                    = {Coarsening dynamics of zero-range processes},
  Author                   = {Godr{\`e}che, Claude and Drouffe, Jean-Michel},
  Journal                  = {Journal of Physics A: Mathematical and Theoretical},
  Year                     = {2016},
  Number                   = {1},
  Pages                    = {015005},
  Volume                   = {50},
  Publisher                = {IOP Publishing}
}

@Article{godreche2005dynamics,
  Title                    = {Dynamics of the condensate in zero-range processes},
  Author                   = {Godr{\`e}che, Claude and Luck, Jean-Marc},
  Journal                  = {Journal of Physics A: Mathematical and General},
  Year                     = {2005},
  Number                   = {33},
  Pages                    = {7215},
  Volume                   = {38},
  Publisher                = {IOP Publishing}
}

@article{grosskinsky2019derivation,
  title={Derivation of mean-field equations for stochastic particle systems},
  author={Grosskinsky, Stefan and Jatuviriyapornchai, Watthanan},
  journal={Stochastic Processes and their Applications},
  volume={129},
  number={4},
  pages={1455--1475},
  year={2019},
  publisher={Elsevier}
}

@article{chleboun2023sizebiased,
      title={Size-biased diffusion limits and the inclusion process}, 
      author={Paul Chleboun and Simon Gabriel and Stefan Grosskinsky},
      year={2023},
      eprint={2304.09722},
      archivePrefix={arXiv},
  journal={arXiv preprint arXiv:2304.09722},
      primaryClass={math.PR}
}

@Article{jatuviriyapornchai2016coarsening,
  Title                    = {Coarsening dynamics in condensing zero-range processes and size-biased birth death chains},
  Author                   = {Jatuviriyapornchai, Watthanan and Grosskinsky, Stefan},
  Journal                  = {Journal of Physics A: Mathematical and Theoretical},
  Year                     = {2016},
  Number                   = {18},
  Pages                    = {185005},
  Volume                   = {49},

  Publisher                = {IOP Publishing}
}

@Article{pakdaman2010fluid,
  Title                    = {Fluid limit theorems for stochastic hybrid systems with application to neuron models},
  Author                   = {Pakdaman, Khashayar and Thieullen, Michele and Wainrib, Gilles},
  Journal                  = {Advances in Applied Probability},
  Year                     = {2010},
  Number                   = {3},
  Pages                    = {761--794},
  Volume                   = {42},

  Publisher                = {Cambridge University Press}
}

@Misc{daipra2017,
  Title                    = {Stochastic mean-field dynamics and applications to life sciences},

  Author                   = {Paolo Dai Pra},
  HowPublished             = {\url{http://www.cirm-math.fr/ProgWeebly/Renc1555/CoursDaiPra.pdf}},
  Note                     = {Accessed 12/07/17},
  Year                     = {2017}
}

@inproceedings{Lovsz2012LargeNA,
  title={Large Networks and Graph Limits},
  author={L{\'a}szl{\'o} Mikl{\'o}s Lov{\'a}sz},
  booktitle={Colloquium Publications},
  year={2012},
  url={https://api.semanticscholar.org/CorpusID:27057044}
}

@article{SGAKtagged,
author = {Stefan Grosskinsky and Angeliki Koutsimpela},
title = {{Tagged particles and size-biased dynamics in mean-field interacting particle systems}},
volume = {30},
journal = {Electronic Communications in Probability},
number = {none},
publisher = {Institute of Mathematical Statistics and Bernoulli Society},
pages = {1 -- 19},
year = {2025},
doi = {10.1214/25-ECP669},
URL = {https://doi.org/10.1214/25-ECP669}
}

@article{Koutsimpela2025MeanField,
  title         = {Mean-field limits in interacting particle systems with symmetric superlinear rates},
  author        = {Koutsimpela, Angeliki},
  journal       = {arXiv preprint arXiv:2509.19617},
  year          = {2025},
  eprint        = {2509.19617},
  archivePrefix = {arXiv},
  primaryClass  = {math.PR},
  doi           = {10.48550/arXiv.2509.19617}
}

@article{Bet_Coppini_Nardi_2024, 
title={Weakly interacting oscillators on dense random graphs}, 
volume={61}, 
DOI={10.1017/jpr.2023.34}, 
number={1}, 
journal={Journal of Applied Probability}, author={Bet, Gianmarco and Coppini, Fabio and Nardi, Francesca Romana}, 
year={2024}, 
pages={255–278}}

@article{oliveira2018interacting,
title={Interacting Diffusions on Random Graphs with Diverging Average Degrees: Hydrodynamics and Large Deviations},
volume={176},
DOI={10.1007/s10955-019-02332-1},
number={5},
journal={Journal of Statistical Physics},
author={Oliveira, Roberto I. and Reis, Guilherme H.},
year={2019},
pages={1057--1087}}

@article{COPPINI2025104728,
title = {Nonlinear Graphon mean-field systems},
journal = {Stochastic Processes and their Applications},
volume = {190},
pages = {104728},
year = {2025},
issn = {0304-4149},
doi = {https://doi.org/10.1016/j.spa.2025.104728},
url = {https://www.sciencedirect.com/science/article/pii/S0304414925001693},
author = {Fabio Coppini and Anna {De Crescenzo} and Huyên Pham}
}

\end{document}